\documentclass[a4paper,12pt]{article}
\usepackage{ayv_paperstyle}
\newcommand{\rdp}{\mathbb{R}^{d+1}}
\newcommand{\tdr}{\mathbb{T}^{d}\times\mathbb{R}}

\newcommand{\intTdRdp}{\int_{\mathbb{T}^d\times \mathbb{R}^{d+1}}}

\newcommand{\intc}[2]{\int_{\mathcal{C}_{#1,#2}}}
\newcommand{\ir}{[0,T]\times\mathbb{T}^d}
\newcommand{\ptd}{\mathcal{P}^1(\mathbb{T}^d)}
\newcommand{\pc}[2]{\mathcal{P}^1(\mathcal{C}_{#1,#2})}
\usepackage{longtable}

\title{Viability analysis of the first-order mean field games}
\author{Yurii Averboukh\footnote{Krasovskii Institute of Mathematics and Mechanics,\ \ e-mail: averboukh@gmail.com}{ }\footnote{Ural Federal University}}
\date{}

\begin{document}

\maketitle

\begin{abstract} In the paper we examine the dependence of the solution of the deterministic mean field game on the initial distribution of players. The main object of study is the mapping which assigns to  the initial time and the initial distribution of players the set  of expected rewards of the representative player corresponding to solutions of mean field game. This mapping can be regarded as a value multifunction. We obtain the sufficient condition for a  multifunction to be a value multifunction. It states that if a multifunction is viable with respect to the dynamics generated by the original mean field game, then it is a value multifunction. Furthermore, the infinitesimal variant of this condition is derived. 
\msccode{91A10, 91A23, 49J52, 49J53, 46G05, 49J21}
\end{abstract}

\section*{Introduction}

The theory of mean field games (MFG) aims to study  noncooperative dynamical games of a large number similar players. The main idea of MFG approach is to  examine the limit case when the number of players tends to infinity and each player becomes negligible.

The concept of mean field games was proposed by Lasry, Lions~\cite{Lasry_Lions_2006_I},~\cite{Lasry_Lions_2006_II},~\cite{Lasry_Lions_2007} and by Huang, Caines, Malham\'{e}~\cite{Huang_Caines_Malhame_2007},~\cite{Huang_Malhame_Caines_2006}. Nowadays, there are several approaches to the mean field game theory. First one reduces the mean field game to the backward-forward system of fully-coupled nonlinear PDEs. The first equation is the Hamilton-Jacobi equation which describes the value function of the representative player. The second equation is the Chapman-Kolmogorov equation and it characterizes the distribution of all players.   Within the framework of this approach the existence and uniqueness problems for mean field games were studied (see~\cite{Gomes_et_book_2017},~\cite{Kolokoltsov_Li_Yang_2011},~\cite{Lasry_Lions_2006_II},~\cite{Lasry_Lions_2007},   and reference therein). Moreover,  one can construct an approximate Nash equilibrium for the finite player game given a solution of the mean field game \cite{Kolokoltsov_Li_Yang_2011}, \cite{lions_lecture}.

The second approach to the mean field games is called probabilistic.  It involves the study of infinite player dynamical game with similar players and mean field interaction among them. The probabilistic approach considers the solution of mean field game as the symmetric Nash equilibrium in this game (see  \cite{Carmona_Delarue_2013},~\cite{Carmona_Delarue_I}, \cite{Carmona_Delarue_II},~\cite{Carmona_Delarue_Lachapelle_2013}, \cite{Carmona_Lacker_2015},~\cite{Lacker_2015}). This allows to  prove that the  open-loop equilibria of finite players games  converge to the solution of mean field game when the number of players tends to infinity~\cite{Fischer_convergence_2017},~\cite{Lacker_convergence_2016}.

The third approach is concerned with the study of the partial differential equation involving the derivatives in the space of probabilities called the master equation. It was proposed by Lions in his seminal lectures \cite{lions_lecture} (see also \cite{Cardaliaguet_notes}). The master equation encapsulates the necessary  information to describe the solution of mean field game. It is used to establish the convergence of feedback equilibia of finite player games to a solution of mean field game a nondegenerate stochastic game \cite{Cardaliaguet_Delarue_Lasry_Lions_2015}. Moreover, Lions noticed that the classical mean field game system is the characteristic system of the master equation. The master equation was discussed in \cite{Bensoussan_Frehse_Yam_book},  \cite{Bensoussan_Frehse_Yam_2014}, \cite{Carmona_master}. Note that in \cite{Carmona_master} the master equation was formally derived using the dynamic programming arguments applied to the optimization problem for the representative player. In~\cite{Kolokoltstov_Troeva_common_2015} approximate equilibria in the finite player game with exogenous noise were constructed based on the solution of the mean field game with common noise.

Nowadays, the existence theorem for the master equation is  obtained for the case of nondegenerate stochastic mean field games (possibly, with common noise)  satisfying Lasry-Lions monotonicity condition \cite{Cardaliaguet_Delarue_Lasry_Lions_2015}. The proof is based on the fact that the usual mean field game system provides characteristics for the master equation.  Furthermore, the short-term existence theorem is proved  in \cite{Carmona_Delarue_II} and \cite{Gangbo_Swiech_2015}.

The paper aims to extend the approach of the master equation to the case when one can not expect the uniqueness of solution to mean field game as well as the continuous dependence on the measure variable. Notice that this situation is quite general \cite{Bardi_Fischer_2017} (see also \cite{Lacker_convergence_2016}). We examine the dependence of   the solution of the mean field game on the initial position  using the viability theory arguments (see \cite{Aubin}, \cite{Aubin_Cellina}). The paper can be considered as an effort to develop a mean field game analog of the viability formulation of the viscosity solution to the Hamilton-Jacobi equations arising in the theory of the zero-sum differential game (see \cite{frankowska}, \cite{Subb_book}, \cite{vinter_wolenski}, \cite{wolenski}).

We restrict our attention to the deterministic mean field game. Moreover, for simplicity, we assume the periodic boundary conditions i.e. the phase space for each player is the $d$-dimensional torus $\td\triangleq \rd/\mathbb{Z}^d$. We follow the framework of the  probabilistic approach. In this case  the study of the deterministic mean field game is reduced to the study of the symmetric of equilibrium of the continuum player game where the dynamics of each agent is \begin{equation}\label{sys:dynamics}
\frac{d}{dt}x(t)=f(t,x(t),m(t),u(t,x(t))).
\end{equation} Here $m(t)$ is a probability on $\td$ describing the distribution of all players at time $t$, $t\in [0,T]$, $x(t)\in\td$, $u(t,x(t))\in U$,  $U$ is a control space.   We assume that the each player  aims to maximize his/her payoff given by
\begin{equation}\label{sys:payoff}
\sigma(x(T),m(T))+\int_{t_0}^Tg(t,x(t),m(t),u(t,x(t)))dt.
\end{equation}

To present the main objective of the paper let us consider the finite player game analogy. If one examine the $N$ player non-cooperative differential game with weakly coupled dynamics of  agents, then the state  space is $(\td)^N$, while the players' outcome is a $N$-dimensional vector. In this finite player game the value function (multifunction)  is the mapping which assigns to each initial position the set of Nash values. 

Turning  to the continuum player game, we get that the phase space should be $(\td)^{\mathfrak{c}}$, while the players' outcome is an element of $\mathbb{R}^\mathfrak{c}$. Here $\mathfrak{c}$ is a continuum set. In the mean field game setting we can reduce the phase space of the game to the set of  probabilities on $\td$. Further, since we are seeking for the symmetric equilibrium, the players starting at the same point get the same outcome. Thus, we can index the players by the points of $\td$ i.e. we put $\mathfrak{c}=\td$. Simultaneously, given a solution of the mean field game, the  expected reward of the representative player depends on his/her initial state continuously. Therefore, the mean field game analog of the value function is the mapping which assigns to an initial time $t_0$ and an initial players' distribution $m_0$ a set of  continuous functions from $\td$ to $\mathbb{R}$. Each element of this set is a value function   for the representative player corresponding to a solution of the mean field game with the initial condition $m(t_0)=m_0$.  

The link between the master equation and the approach developed in the paper is as follows. If the function $\varphi(t,x,m)$ solves the master equation for the mean field game, then the  mapping $(t,m)\mapsto \{\varphi(t,\cdot,m)\}$ is a value multifunction in the sense of the paper. Conversely,   under certain regularity conditions those include the uniqueness of the  solution of the mean field game, the value multifunction is single-valued and  provides a solution of the master equation for mean field game. However, as it was mentioned above, we can not restrict our attention to single-valued functions due to  the multiplicity of the solutions of the mean field games \cite{Bardi_Fischer_2017}.

To examine the value multifunction we introduce a family of set-valued mappings  $\Psi^{r,s}:\mathcal{P}^1(\td)\times C(\td)\rightrightarrows \mathcal{P}^1(\td)\times C(\td)$   ($s,r\in [0,T]$, $r\geq s$). Here $\ptd$ stands for the set of probabilities on $\td$. Given  $m\in \mathcal{P}^1(\td)$, $\phi\in C(\td)$, $\Psi^{r,s}(m,\phi)$ is the set of pairs $(\mu,\psi)$ such that $\phi$ is a reward of the representative player corresponding to a solution of the mean field game on $[s,r]$ with  dynamics (\ref{sys:dynamics}), the initial condition $m(s)=m$ and the payoff functional of each agent given by 
$$\psi(x(r))+\int_s^rg(t,x(t),m(t),u(t,x(t)))dt, $$ while the probability $\mu$ is the corresponding distribution of agents at the time $r$. The family of transforms $\{\Psi^{r,s}\}$ determines a forward dynamics on $\mathcal{P}^1(\td)\times C(\td)$. Below we call it a mean field game dynamics. Apparently, the notion of mean field game dynamics is  close to the forward-forward mean field games studied in \cite{Gomes_forward}.  

We prove that if a multifunction is viable with respect to the mean field game dynamics, then this multifunction is a value multifunction. Furthermore, we study the infinitesimal form of the  proposed viability condition. To this end we introduce the set-valued derivative of the multifunction $\mathcal{V}:[0,T]\times\mathcal{P}^1(\td)\rightrightarrows C(\td)$ by virtue of the mean field game dynamics. The viability theorem proved in the paper states that the multifunction $\mathcal{V}$ is viable with respect to the dynamics  if and only if the set-valued derivative is nonempty at any point of the graph of $\mathcal{V}$.

The paper is organized as follows. General notation and assumption are introduced in Section~\ref{sect:preliminaries}. Moreover, in this section we give some properties of the dynamics of distribution of players.  Section~\ref{sect:solution_def} is concerned with the definition of solution to the first-order mean field game. In Section~\ref{sect:value_multifunction} we introduce the notion of the value multifunction and formulate the sufficient condition for a given multifucntion of  time and probability to be a value multifunction. The condition involves the viability property with respect to  the mean field game dynamics. The viability theorem which provides the infinitesimal form of the viability property  is introduced in Section~\ref{sect:viability_theorem}. The subsequent sections are devoted to the proof of this theorem. Auxiliary lemmas are given in Section~\ref{sect:auxiliary}. The sufficiency and necessity parts of the viability theorem are proved in Sections~\ref{sect:proof_sufficiency}~and~\ref{sect:proof_necissity} respectively. The proof of  some properties of  dynamical systems in the space of probability measures 
is given in Appendix A. Finally, Appendix B provides the list of main notation.

\section{Preliminaries}\label{sect:preliminaries}
\subsection{General notations}\label{subsect:notation}
If $(X,\rho_X)$ is a separable metric space, $\Upsilon\subset X$, $x\in X$, then put
$$\mathrm{dist}(x,\Upsilon)\triangleq \inf\{\rho_X(x,y):y\in\Upsilon\}. $$  Below $\mathcal{B}(X)$ stands for the Borel $\sigma$-algebra on $X$. We denote by $\mathcal{P}(X)$ the set of all Borel probabilities on $X$.
Further, let $\mathcal{P}^1(X)$ stand for the set of probabilities $m$ on $X$ such that, for some $x_*\in X$,
$$\int_X\rho_X(x,x_*)m(dx)<\infty. $$ We endow $\mathcal{P}^1(X)$ with the Kantorovich-Rubinstein metric (1-Wasserstein metric) defined by the rule: for $m_1,m_2\in \mathcal{P}^1(X)$,
\begin{equation}\label{intro:wass_distance}
\begin{split}
W_1(m_1,m_2) &\triangleq\inf\left\{\int_{X\times X}\rho_X(x_1,x_2)\pi(d(x_1,x_2)):\pi\in \Pi(m_1,m_2)\right\} \\
&=\sup\left\{\int_X \phi(x) m_1(dx)-\int_X\phi(x)m_2(dx):\phi\in{\rm Lip}_1(X)\right\}.
\end{split}
\end{equation} Here $\Pi(m_1,m_2)$ is the set of probabilities $\pi$ on $X\times X$ such that its marginal distributions are $m_1$ and $m_2$ respectively, i.e.
$$\Pi(m_1,m_2)=\{\pi\in\mathcal{P}(X\times X):\pi(\Upsilon\times X)=m_1(\Upsilon),\ \ \pi(X\times \Upsilon)=m_2(\Upsilon)\}; $$
${\rm Lip}_K(X)$ denotes the set of $K$-Lipschitz continuous functions $\phi:X\rightarrow \mathbb{R}$. Note that, if $X$ is compact, then the Wasserstein distance metrizes the narrow convergence and $\mathcal{P}^1(X)$ is also compact.

If $(\Omega^1, \Sigma^1)$ and $(\Omega^2, \Sigma^2)$ are measurable spaces, $h:\Omega^1\mapsto\Omega^2$ is measurable, $m$ is a probability of $\Sigma^1$, then denote by $h_\# m$ the probability on $\Sigma^2$ given by the rule: for $\Upsilon\in \Sigma^2$,
\begin{equation}\label{intro:push_forward}
(h_\# m)(\Upsilon)\triangleq m(h^{-1}(\Upsilon)). 
\end{equation}

If $(X,\rho_X)$, $(Y,\rho_Y)$ are separable metric spaces,  $\pi\in\mathcal{P}(X\times Y)$, then denote by $\pi(\cdot|x)$ a conditional distribution on $Y$ given $x$ i.e., for each $x$, $\pi(\cdot|x)$ is a probability on $Y$ and, for any $\varphi\in C_b(X\times Y)$,
\begin{equation}\label{intro:disintegration}
\int_{X\times Y}\varphi(x,y)\pi(d(x,y))=\int_X\int_Y\varphi(x,y)\pi(dy|x)m(dx).  
\end{equation} Here $m$ is a marginal distribution on $X$ of $\pi$.

Now, let $(X,\rho_X)$, $(Y,\rho_Y)$ and $(Z,\rho_Z)$ be separable metric spaces, and let $\pi_{1,2}\in\mathcal{P}(X\times Y)$ and $\pi_{2,3}\in\mathcal{P}(Y\times Z)$ have the same marginal distribution on $Y$ equal to $m$. We define the probability $\pi_{1,2}*\pi_{2,3}\in\mathcal{P}(X\times Z)$  by the following rule: for $\varphi\in C_b(X\times Z)$,
\begin{equation}\label{intro:composition}
\int_{X\times Y}\varphi(x,z)(\pi_{1,2}*\pi_{2,3})(d(x,z))
\triangleq   \int_{Y}\int_X\int_Z\varphi(x,z)\pi_{2,3}(dz|y)\pi_{1,2}(dx|y)m(dy).
\end{equation} The probability $\pi_{1,2}*\pi_{2,3}$ is the composition of $\pi_{1,2}$ and $\pi_{2,3}$. Note that in \cite{Ambrosio} it is denoted by $\pi_{2,3}\circ\pi_{1,2}$.

If $\mathcal{Y}$ is a multifunction from $[s,r]$ to finite dimensional euclidean space, then $\int_s^r\mathcal{Y}(t)dt$ stands for the Aumann integral i.e.
$$\int_s^r\mathcal{Y}(t)dt\triangleq \left\{\int_s^r y(t)dt:y(\cdot)\text{ integrable, }y(t)\in \mathcal{Y}(t)\text{ a.e. }\right\}. $$

\subsection{Probabilities on state space and on space of motions}\label{subsect:state}

As it was mentioned above, we assume that the phase space is $d$-dimensional torus   $\mathbb{T}^d=\mathbb{R}^d/\mathbb{Z}^d$.  Further, $\mathbb{T}^d\times\mathbb{R}$ is an extended phase space; $\mathrm{p}$ denotes a natural projection from $\mathbb{T}^d\times\mathbb{R}$ onto $\mathbb{T}^d$. 
If $(x,z)\in \mathbb{T}^d\times\mathbb{R}$, then
$$\|(x,z)\|\triangleq \|x\|+|z|.$$

Notice that the set $\mathcal{C}_{s,r}\triangleq C([s,r],\tdr)$ is the set of motions on $[s,r]$ in the extended space. 

Let $e_t:\mathcal{C}_{s,r}\rightarrow \td$ and $\hat{e}_t:\mathcal{C}_{s,r}\rightarrow \tdr$ be evaluation operators defined by the rules: if $w(\cdot)=(x(\cdot),z(\cdot))\in\mathcal{C}_{s,r}$, then \begin{equation}\label{intro:evaluation_operators}
e_t(w(\cdot))\triangleq x(t),\ \ \hat{e}_t(w(\cdot))\triangleq w(t).
\end{equation} Notice  that $e_t=\mathrm{p}\circ \hat{e}_t$.

Furthermore, if $s,r\in [0,T]$, $s<r$, $t\in [s,r]$, $\chi_1,\chi_2\in\pc{s}{r}$, then
\begin{equation}\label{ineq:Wasserstein_e_t}
W_1(e_t{}_\#\chi_1,e_t{}_\#\chi_2)\leq W_1(\hat{e}_t{}_\#\chi_1,\hat{e}_t{}_\#\chi_2)\leq W_1(\chi_1,\chi_2).
\end{equation}

Denote the set of all measurable functions from  $[s,t]$ to $\ptd$ by $\mathcal{M}_{s,t}$. Analogously, let $\mathcal{N}_{s,r}$ stand for the set of all measurable functions defined on $[s,r]$ taking values in $\mathcal{P}^1(\tdr)$. We will call elements of both $\mathcal{M}_{s,r}$ and $\mathcal{N}_{s,r}$ flows of probabilities. 

For $m\in\ptd$,  denote by $\widehat{m}$ its lifting to $\mathcal{P}^1(\tdr)$ defined by the rule: for $\varphi\in C_b(\tdr)$,
\begin{equation}\label{intro:lifting_m}
\int_{\tdr}\varphi(x,z)\widehat{m}(d(x,z))\triangleq \int_{\td}\varphi(x,0)m(dx).
\end{equation}

If $\phi\in C(\td)$ and $\nu\in \mathcal{P}^1(\tdr)$, then denote by $[\phi,\nu]$ the averaging of the function $(x,z)\mapsto \phi(x)+z$ according to $\nu$ i.e.
\begin{equation}\label{intro:action}
[\phi,\nu]\triangleq \int_{\tdr}(\phi(x)+z)\nu(d(x,z)).
\end{equation} 
\begin{remark}\label{remark:formula}
	If $\nu=\hat{e}_s{}_\#\chi$, then
	$$[\phi,\nu]=\int_{\mathcal{C}_{s,r}}(\phi(x(s))+z(s))\chi(d(x(\cdot),z(\cdot))). $$
\end{remark}

The definition of $[\phi,\nu]$ implies that, if $\phi,\phi'\in\mathrm{Lip}_K(\td)$, $\nu,\nu'\in\mathcal{P}^1(\tdr)$, then
\begin{equation}
\label{estima:action_continuous}
\Bigl|[\phi,\nu]-[\phi',\nu']\Bigr|\leq \|\phi-\phi'\|+KW_1(\nu,\nu').
\end{equation}

Now let us introduce the notion of concatenations of the  probabilities on  the set of motions. First, we recall the notion of concatenation of motions. Let $s<r<\theta$. If $w_1(\cdot)\in\mathcal{C}_{s,r}$, $w_2(\cdot)\in\mathcal{C}_{r,\theta}$ are such that $w_1(r)=w_2(r)$, then $w_1(\cdot)\odot w_2(\cdot)$ is a motion $w(\cdot)\in\mathcal{C}_{s,\theta}$ given by
\begin{equation}\label{intro:concatination_mothions}
w(t)=\left\{\begin{array}{cc}
w_1(t), & t\in [s,r], \\
w_2(t), & t\in [r,\theta].
\end{array}\right. 
\end{equation}
Now, let $\chi_1\in\mathcal{P}^1(\mathcal{C}_{s,r})$, $\chi_2\in\mathcal{P}^1(\mathcal{C}_{r,\theta})$ be such that $\hat{e}_r{}_\#\chi_1=\hat{e}_r{}_\#\chi_2$. Let $\chi_2(d(w(\cdot))|w_0)$ denote the disintegration of $\chi_2$ along $\hat{e}_r{}_\#\chi_2$. Define the concatenation of probabilities $\chi_1$ and $\chi_2$ $\chi\triangleq\chi_1\odot\chi_2$ by the following rule: for any $\varphi\in C_b(\mathcal{C}_{s,\theta})$,
\begin{equation}\label{intro:concatination}
\int_{\mathcal{C}_{s,\theta}}\varphi(w(\cdot))\chi(d(w(\cdot)))=
\int_{\mathcal{C}_{s,r}}\int_{\mathcal{C}_{r,\theta}}\varphi(w_1(\cdot)\odot w_2(\cdot))\chi_2(d(w_2(\cdot))|w_1(r))\chi_1(d(w_1(\cdot))).
\end{equation}

\subsection{Dynamics of distribution of players}\label{subsect:dynamics}
We assume that the set $U$ and the functions $f$, $g$, $\sigma$ satisfy the following assumptions.
\begin{list}{(M\arabic{tmp})}{\usecounter{tmp}}
	\item $U$ is a metric compact;
	\item functions $f$, $g$ and $\sigma$ are continuous;
	\item\label{cond:alpha_cont} there exists a function $\alpha:\mathbb{R}\rightarrow [0,+\infty)$ such that $\alpha(\delta)\rightarrow 0$ as $\delta\rightarrow 0$ and, for any $t',t''\in [0,T]$, $x\in\td$, $m\in\ptd$, $u\in U$, $$\|f(t',x,m,u)-f(t'',x,m,u)\|\leq\alpha(t'-t''),$$
	$$|g(t',x,m,u)-g(t'',x,m,u)|\leq\alpha(t'-t'').$$
	\item\label{cond:lipschitz} $f$ and $g$ are Lipschitz continuous with respect to the space variable $x$ and probability $m$, i.e. there exists a constant $L$ such that, for any $t\in [0,T]$, $x',x''\in \mathbb{T}^d$, $m',m''\in\ptd$, $u\in U$,
	$$\|f(t,x',m',u)-f(t,x'',m'',u)\|\leq L(\|x'-x''\|+W_1(m',m'')), $$
	$$|g(t,x',m',u)-g(t,x'',m'',u)|\leq L(\|x'-x''\|+W_1(m',m''));$$
	\item\label{cond:lip+sigma} there exists a constant $\varkappa$ such that, for any $x',x''\in\mathbb{T}^d$ and any $m\in\ptd$,
	$$|\sigma(x',m)-\sigma(x'',m)|\leq \varkappa\|x'-x''\|. $$
	
\end{list}

Conditions (M1)--(M5) imply the existence of a constant $R$ such that
\begin{equation}\label{intro:R}
\|f(t,x,m,u)\|,|g(t,x,m,u)|\leq R
\end{equation}

Using the relaxation of the control problem for the representative player (see~\cite{Warga}), we get that his/her dynamics in the extended phase space obeys the following differential inclusion:
\begin{equation}\label{incl:F}
(\dot{x}(t),\dot{z}(t))\in F(t,x(t),m(t)).  
\end{equation} Here
\begin{equation}\label{intro:F}
F(t,x,m)\triangleq \mathrm{co}\{(f(t,x,m,u),g(t,x,m,u)):u\in U\}.
\end{equation}
An equivalent approach to the relaxation of control problems is based on measure-valued controls \cite{Warga}. 
A measure-valued control is a function $\xi:[s,r]\times\mathcal{B}(U)\rightarrow [0,1]$ satisfying the following conditions:
\begin{itemize}
	\item for each $t\in [s,r]$, $\xi(t,\cdot)$ is a probability on $U$;
	\item for any $\varphi\in C(U)$, the functions $$t\mapsto\int_U\varphi(u)\xi(t,du)$$ is measurable.
\end{itemize} We denote the set of measure-valued controls on $[s,r]$ by $\mathcal{U}_{s,r}$. Notice (see \cite{Warga}) that under conditions (M1)--(M5), if $(x(\cdot),z(\cdot))\in \mathcal{C}_{s,r}$ satisfies differential inclusion (\ref{incl:F}), then there exists $\xi\in\mathcal{U}_{s,r}$ such that, for a.e. $t\in [s,r]$,
\begin{equation}\label{eq:xi_F}
\dot{x}(t)= \int_U f(t,x(t),m(t),u)\xi(t,du),\ \ \dot{z}(t)= \int_U g(t,x(t),m(t),u)\xi(t,du).
\end{equation}

If $s,r\in [0,T]$, $s<r$, $y\in \td$, $m(\cdot)\in\mathcal{M}_{s,r}$, then denote by $\mathrm{Sol}(r,s,y,m(\cdot))$ the set of solution of (\ref{incl:F}) satisfying $x(s)=y$.
Further, put
\begin{equation}\label{introl:SOL}
\mathrm{SOL}(r,s,m(\cdot))\triangleq \bigcup_{y\in\td}\mathrm{Sol}(r,s,y,m(\cdot)). 
\end{equation}

Notice that, for any $s,r\in [0,T]$, $m(\cdot)\in \mathcal{M}_{s,r}$, each $(x(\cdot),z(\cdot))\in\mathrm{SOL}(r,s,m(\cdot))$ is absolutely continuous and 
\begin{equation}\label{estima:dot_x_dot_z}
\|\dot{x}(t)\|,\ \ |\dot{z}(t)|\leq R\mbox{ a.e. }t\in [s,r].
\end{equation}

Integrating (\ref{incl:F}), we get that the dynamics  of distribution of players in the extended space can be described by the following mean field type differential inclusion:
\begin{equation}\label{system:mfdi}
\frac{d}{dt}\nu(t)\in \langle \widehat{F}(t,\cdot,\nu(t)),\nabla\rangle \nu(t), 
\end{equation} where $\widehat{F}(t,w,\nu)$ is defined by the rule
\begin{equation}\label{intro:hat_F}
\widehat{F}(t,w,\nu)\triangleq \mathrm{co}\{(f(t,\mathrm{p}(w),\mathrm{p}_\#\nu,u),f(t,\mathrm{p}(w),\mathrm{p}_\#\nu,u)):u\in U\}. 
\end{equation}

In the general form the notion of solution to the mean field type differential inclusion can be introduced as follows.
\begin{definition}\label{def:solution_mfdi} Let $X$ be a finite dimensional Euclidean space. Further, 
	let $G(t,w,\nu)$ be a multivalued function defined on $[0,T]\times X\times\mathcal{P}^1(X)$ with values in $X$. We say that the function $[s,r]\ni t\mapsto \nu(t)\in\mathcal{P}^1(X)$ solves the mean field type differential inclusion (shortly, MFDI)
	$$\frac{d}{dt}\nu(t)\in \langle G(t,\cdot,\nu(t)),\nabla\rangle \nu(t), $$ if there exists a probability $\chi\in \mathcal{P}^1(C[s,r],X)$ such that $\nu(t)$ is an evaluation of $\chi$ at time $t$, and  $\chi$-a.e. $w(\cdot)\in C([s,r],X)$ satisfies the differential inclusion
	$$\frac{d}{dt}w(t)\in G(t,w(t),\nu(t)). $$
\end{definition}

\begin{remark}\label{remark:mfdi}
	Notice that  $\nu(\cdot)$ solves mean field type differential inclusion (\ref{system:mfdi}) on $[s,r]$ if and only if there exists $\chi\in \mathcal{P}^1(\mathcal{C}_{s,r})$ such that $\nu(t)=\hat{e}_t{}_\#\chi$ and $\mathrm{supp}(\chi)\subset \mathrm{SOL}(r,s,m(\cdot))$ for $m(t)=\mathrm{p}_\#\nu(t)$.
\end{remark}
Using the same methods as in \cite{Sznitman}, one can prove the existence of at least one solution to (\ref{system:mfdi}) satisfying $\nu(s)=\nu_*$.

Now let us list the properties of the solutions to mean field type differential inclusion (\ref{system:mfdi}). They are proved in the Appendix A.
\begin{proposition}\label{prop:shift}
	Let $\nu(\cdot)\in \mathcal{N}_{s,r}$ solve (\ref{system:mfdi}), and let $\nu_*\in\mathcal{P}^1(\tdr)$ be such that $\mathrm{p}_\#\nu(s)=\mathrm{p}_\#\nu_*$.  Then there exists a flow of probabilities $\bar{\nu}(\cdot)\in \mathcal{N}_{s,r}$ such that
	\begin{enumerate}
		\item $\bar{\nu}(\cdot)$ solves (\ref{system:mfdi});
		\item $\mathrm{p}_\#\nu(t)=\mathrm{p}_\#\bar{\nu}(t)$;
		\item $\nu(s)=\nu_*$;
		\item for any $\phi\in C(\td)$, and any $t\in [s,r]$, 
		$$[\phi,\bar{\nu}(t)]=[\phi,\nu(t)]-\int_{\tdr}z\nu(s,d(x,z))+\int_{\tdr}z\nu_*(d(x,z)). $$
	\end{enumerate} 
\end{proposition}

\begin{proposition}\label{prop:concatination}
	If $s_0,s_1,s_2\in [0,T]$, $s_0<s_1<s_2$, $\nu_1(\cdot)\in \mathcal{N}_{s_0,s_1}$, $\nu_2\in\mathcal{N}_{s_1,s_2}$ solve (\ref{system:mfdi}) and $\nu_1(s_1)=\nu_2(s_2)$, then
	$$\nu(t)\triangleq \left\{\begin{array}{cc}
	\nu_1(t), & t\in [s_0,s_1];\\
	\nu_2(t), & t\in [s_1,s_2]
	\end{array}\right. $$ is also a solution to (\ref{system:mfdi}).
\end{proposition}

\begin{proposition}\label{prop:limit} Let $\{\tau_i\}_{i=1}^\infty\subset [s,r]$, $\tau\in [s,r]$, $\{\nu_i(\cdot)\}_{i=1}^\infty\subset \mathcal{N}_{s,r}$, $\{\nu^\natural_i\}_{i=1}^\infty\subset \mathcal{P}^1(\tdr)$, $\nu^\natural\in\mathcal{P}^1(\tdr)$. Assume that, for each $i$, $\nu_i(\cdot)$ solves (\ref{system:mfdi}) and satisfies $\nu_i(\tau_i)=\nu_i^\natural$. Additionally, assume that $\tau_i\rightarrow\tau$, $W_1(\nu_i^\natural,\nu^\natural)\rightarrow 0$ as $i\rightarrow\infty$. Then there exist a sequence $\{i_k\}$ and a flow of probabilities $\nu^*(\cdot)\in\mathcal{N}_{s,r}$ such that
	\begin{enumerate}
		\item $\nu^*(\cdot)$ solves (\ref{system:mfdi});
		\item $\nu^*(\tau)=\nu^\natural$;
		\item $$\lim_{k\rightarrow \infty}\sup_{t\in [s,r]}W_1(\nu_{i_k}(t),\nu^*(t))=0. $$
	\end{enumerate}
	
\end{proposition}

\section{Solution of the first-order mean field game}\label{sect:solution_def}
As it was mentioned in the Introduction, there are several methods to analyze  mean field game (\ref{sys:dynamics}), (\ref{sys:payoff}). The approach based on PDEs reduces the original problem to the mean field game system
$$\frac{\partial V}{\partial t}+H(t,x,m(t),\nabla V)=0, V(T,x)=\sigma(x,m(T)),  $$
$$\frac{d}{dt}m(t)=\left\langle\frac{\partial H(t,x,m(t),\nabla V)}{\partial p},\nabla\right\rangle m(t),\ \ m(t_0)=m_0. $$ Here, for $t\in [0,T]$, $x\in \td$, $m\in\mathcal{P}^1(\td)$, $p\in\rd$, $$H(t,x,m,p)\triangleq \max_{u\in U}[\langle p,f(t,x,m,u)\rangle+g(t,x,m,p)]. $$ Within the framework of this approach a solution of the mean field game is defined as a solution of this system. However, for our purposes it is convenient to use the probabilistic approach.  The link between the mentioned approached is discussed in~\cite{Lacker_2015}.

We adapt the probabilistic approach for the first order mean field game. The following definition is close to one proposed in~\cite{Averboukh_2015}.

\begin{definition}\label{def:solution}
	We say that a pair $(V,m(\cdot))$, where $V:[t_0,T]\times\td\rightarrow \mathbb{R}$ is a continuous function and $m(\cdot)\in\mathcal{M}_{t_0,T}$, is a solution to  mean field game (\ref{sys:dynamics}), (\ref{sys:payoff}), if there exists a probability $\chi\in \mathcal{P}^1(\mathcal{C}_{t_0,T})$ such that
	\begin{enumerate}
		\item $m(t)=e_t{}_\#\chi$;
		\item $V(s,y)$ is a value of the optimization problem
		\begin{equation*}\label{represet_player:problem}
		\text{maximize }\left[\sigma(x(T),m(T))+z(T)-z(s) \right]
		\end{equation*}
		\begin{equation*}\label{represet_player:dynamics}
		\text{subject to }(x(\cdot),z(\cdot))\in\mathrm{Sol}(T,s,y,m(\cdot)); 
		\end{equation*}
		\item $\mathrm{supp}(\chi)\subset \mathrm{SOL}(T,t_0,m(\cdot))$;
		\item for any $s,r\in [t_*,T]$, $s<r$, and any $(x(\cdot),z(\cdot))\in \mathrm{supp}(\chi)$,
		$$V(s,x(s))+z(s)=V(r,x(r))+z(r). $$
	\end{enumerate}
\end{definition}

Given $m_0\in\ptd$, there exists at least one solution $(V,m(\cdot))$ to mean field game (\ref{sys:dynamics}), (\ref{sys:payoff}) satisfying $m(t_0)=m_0$~\cite{Averboukh_2015}. However, the uniqueness result is not valid in the general case~\cite{Bardi_Fischer_2017}.

Now let us introduce the equivalent formulation of Definition~\ref{def:solution} using the notion of solution to the differential inclusion. To this end, for $s,r\in [0,T]$, $s\leq r$, $m(\cdot)\in\mathcal{M}_{s,r}$, define the operator $B^{s,r}_{m(\cdot)}:C(\td)\rightarrow C(\td)$ by the rule:
\begin{equation}\label{intro:Bellman}
(B^{s,r}_{m(\cdot)}\psi)(y)\triangleq\{\psi(x(r))+z(r)-z(s):(x(\cdot),z(\cdot))\in\mathrm{Sol}(r,s,y,m(\cdot))\}. 
\end{equation}
The family $\{B_{m(\cdot)}^{s,r}\}_{s\leq r}$ is a backward propagator. Indeed,  if $ m(\cdot)$ is a flow of probabilities on $[s,\theta]$, $r\in [s,\theta]$, then  
\begin{equation}\label{equality:B_semogroup}
B^{s,r}_{m(\cdot)} B^{r,\theta}_{m(\cdot)}=B_{m(\cdot)}^{s,\theta}. 
\end{equation} Moreover, $B_{m(\cdot)}^{s,s}=\mathrm{Id}$. Notice that, for any $s,r\in [0,T]$, $s\leq r$, the operator $B^{s,r}_{m(\cdot)}$ is linear in max-plus algebra.

\begin{proposition}\label{prop:equivalence} The pair $(V,m(\cdot))$ is a solution to mean field differential game~(\ref{sys:dynamics}),~(\ref{sys:payoff}) if and only if there exists a solution to MFDI (\ref{system:mfdi}) on $[t_0,T]$ $\nu(\cdot)$ such that, for any $s\in [t_0,T]$, 
	\begin{enumerate}
		\item $m(s)=\mathrm{p}_\#\nu(s)$, $s\in [t_0,T]$;
		\item $V(s,\cdot)=B_{m(\cdot)}^{s,T}\sigma(\cdot,m(T))$;
		\item $[\sigma(\cdot,m(T)),\nu(T)]\geq [V(s,\cdot),\nu(s)]$.
	\end{enumerate}
\end{proposition}
\begin{proof} 
	By Remark \ref{remark:mfdi} we have that  conditions 1 and 3 of Definition \ref{def:solution} are equivalent to the fact that $\nu(\cdot)$ solves~(\ref{system:mfdi}) and $m(s)=\mathrm{p}_\#\nu(t)$. Clearly, condition 2 of Definition \ref{def:solution} can be rewritten in the form $V(s,\cdot)=B_{m(\cdot)}^{s,T}\sigma(\cdot,m(T))$. It remains to show that Condition 4 of Definition \ref{def:solution} and Condition 3 of the proposition are equivalent. Let $\chi$ be such that $\nu(t)=\hat{e}_t{}_\#\chi$ and $\operatorname{supp}(\chi)\subset \operatorname{SOL}(r,s,m(\cdot))$. To prove the first implication integrate condition 4 of Definition \ref{def:solution}, for $r=T$. Thus, using Remark \ref{remark:mfdi}, we get
	\begin{equation*}\begin{split}
	[\sigma(\cdot,m(T)),\nu(T)]&= \int_{\mathcal{C}_{s,T}}[\sigma(x(T),m(T))+z(T)]\chi(d(x(\cdot),z(\cdot)))\\&=
	\int_{\mathcal{C}_{s,T}}[V(s,x(s))+z(s)]\chi(d(x(\cdot),z(\cdot)))= [V(s,\cdot),\nu(s)].
	\end{split} 
	\end{equation*}
	Conversely, since $V(s,\cdot)=B_{m(\cdot)}^{s,T}\sigma(\cdot,m(T))$, we have that, for any $(x(\cdot),z(\cdot))\in \mathrm{SOL}(T,t_0,m(\cdot))$,
	$$V(s,x(s))+z(s)\geq \sigma(x(T),m(T))+z(T). $$ Further, using Remark \ref{remark:formula}, one can rewrite  Condition 3 of the proposition in the form
	\begin{equation*}
	\int_{\mathcal{C}_{s,T}}[\sigma(x(T),m(T))+z(T)]\chi(d(x(\cdot),z(\cdot)))\geq
	\int_{\mathcal{C}_{s,T}}[V(s,x(s))+z(s)]\chi(d(x(\cdot),z(\cdot))).
	\end{equation*}
	Therefore, we conclude that, for $\chi$-a.e. $(x(\cdot),z(\cdot))\in\mathrm{SOL}(T,t_0,m(\cdot))$,
	$$ \sigma(x(T),m(T))+z(T)=V(s,x(s))+z(s).$$ This and continuity of the functions $\sigma$ and $V$ imply Condition 4 of Definition \ref{def:solution}.
\end{proof}

In the end of this section let us present the continuity properties of the propagator $B$. These properties will be widely used below. 
In the following the constant $R$ satisfies (\ref{intro:R}).

\begin{lemma}\label{lm:lipschitz}
	Assume that $m(\cdot)\in\mathcal{M}_{0,r}$, $K>0$, $\psi\in \mathrm{Lip}_K(r)$ and $\varphi\in C([0,r]\times \td)$ is such that, for all $s\in [0,r]$, $\varphi(s,\cdot)=B^{s,r}_{m(\cdot)}\psi$. Then
	\begin{enumerate}
		\item for $s\in [0,T]$, $\varphi(s,\cdot)\in\mathrm{Lip}_{(K+1)e^{L(r-s)}-1}(\mathbb{T}^d)$;
		\item for any $s,s'\in [0,r]$, $|\varphi(s,y)-\varphi(s',y)|\leq  R(K+1)e^{L(r-\max\{s',s\})}|s-s'|. $
	\end{enumerate}
\end{lemma}
\begin{proof}
	Let $y_1,y_2\in\mathbb{T}^d$. Since $\varphi(s,\cdot)=B^{s,r}_{m(\cdot)}\psi$, there exist $(x_1(\cdot),z_1(\cdot))\in\mathcal{C}_{s,r}$ and $\xi\in\mathcal{U}_{s,T}$ such that $(x_1(\cdot),z_1(\cdot))$ solves differential inclusion (\ref{incl:F}),   $x_1(s)=y_1$,
	$$\phi(s,y_1)=\psi(x_1(r))+z_1(r)-z_1(s), $$ and (\ref{eq:xi_F}) holds  for $x(\cdot)=x_1(\cdot)$, $z(\cdot)=z_1(\cdot)$ and $\xi=\xi_1$.
	
	Let $(x_2(\cdot),z_2(\cdot))$ solve initial the value problem for (\ref{eq:xi_F}) with $\xi=\xi_1$ and $x_2(s)=y_2$, $z_2(s)=z_1(s)$. We have that
	$$\|x_1(t)-x_2(t)\|\leq \|y_1-y_2\|+ L\int_s^t\|x_1(\tau)-x_2(\tau)\|d\tau. $$ Using Gronwall's inequality we get
	\begin{equation}\label{estima_Lip_x:pr_lipsch}
	\|x_1(t)-x_2(t)\|\leq \|y_1-y_2\|e^{L(t-s)}.
	\end{equation}
	Further,
	\begin{equation}\label{estima_Lip_z:pr_lipsch}|z_1(t)-z_2(t)|\leq \|y_1-y_2\|(e^{L(t-s)}-1).\end{equation}
	Since $(x_2(\cdot),z_2(\cdot))$  solves (\ref{incl:F}) and $\varphi(s,\cdot)=B^{s,r}_{m(\cdot)}\psi$, we have that
	$$\varphi(s,y_2)\geq \psi(x_2(r))+z_2(r)-z_1(s). $$ Combining this, (\ref{estima_Lip_x:pr_lipsch}) and (\ref{estima_Lip_z:pr_lipsch}) we conclude that
	$$\varphi(s,y_2)-\varphi(s,y_1)\geq -[(K +1)e^{L(r-s)}-1]\|y_1-y_2\|. $$ The opposite inequality is proved in the same way.

	Now, we turn to the second statement of the Lemma. Without loss of generality assume that $s'<s$. The semigroup property for the operator $B$ (see (\ref{equality:B_semogroup})) implies that $\varphi(s',\cdot)=B^{s',s}_{m(\cdot)}\varphi(s,\cdot)$. 
	Thus, for each $y\in\td$, there exists a trajectory  $(x(\cdot),z(\cdot))\in\mathrm{Sol}(s,s',y,m(\cdot))$ such that
	$$\varphi(s',y)=\varphi(s,x(s))+z(s)-z(s'). $$ Hence,
	\begin{equation*}\begin{split}
	|\varphi(s',y)-\varphi(s,y)|&=|\varphi(s,x(s))-\varphi(s,y)+z(s)-z(s')|\\&\leq |\varphi(s,x(s))-\varphi(s,y)|+|z(s)-z(s')|. 
	\end{split}
	\end{equation*} By the first statement of the Lemma and inequality (\ref{estima:dot_x_dot_z}) we have that
	$$|\varphi(s',y)-\varphi(s,y)|\leq R(K+1)e^{L(r-\max\{s',s\})}(s-s'). $$ This completes the proof of the second statement of the Lemma.
	
\end{proof}

\begin{lemma}\label{lm:B_continuity}
	If $s,r,r'\in [0,T]$, $s<r\leq r'$, $m(\cdot),m'(\cdot)\in\mathcal{M}_{s,r'}$, $K>0$, $\psi,\psi'\in \mathrm{Lip}_{K}(\mathbb{T}^d)$, then 
	\begin{equation*}
	\begin{split}
	\|B^{s,r}_{m(\cdot)}\psi-B^{s,r'}_{m'(\cdot)}\psi'\|\leq \|&\psi'-\psi\|+ (K+1)R(r'-r)\\&+L(Ke^{LT}+LTe^{LT}+1)(r-s)\sup_{t\in [s,r]}W_1(m'(t),m(t)).\end{split}\end{equation*}
\end{lemma}
\begin{proof}
	Let $y\in\td$ and let $(x_1(\cdot),z_1(\cdot))\in \mathrm{Sol}(r,s,y,m(\cdot))$ be such that
	$$(B^{s,r'}_{m'(\cdot)}\psi')(y)=\psi'(x_1(r'))+z_1(r')-z_1(s). $$
	
	There exists a relaxed control $\xi\in\mathcal{U}_{s,r'}$ such that (\ref{eq:xi_F}) holds true for $x(\cdot)=x_1(\cdot)$, $z(\cdot)=z_1(\cdot)$. Let $x_2(\cdot)$ solve the initial value problem
	$$ \frac{d}{dt}x_2(t)=\int_U f(t,x_2(t),m(t),u)\xi(t,du),\ \ x_2(s)=y.$$ Further, put
	$$z_2(t)\triangleq z_1(s)+\int_s^r\int_U g(t,x_2(t),m(t),u)\xi(t,du)dt. $$
	We have that
	$$\|x_1(t)-x_2(t)\|\leq L\int_s^t\|x_1(\theta)-x_2(\theta)\|d\theta+L(t-s)\sup_{\theta\in [s,r']}W_1(m(\theta),m'(\theta)). $$
	Using Gronwall's inequality, we get
	\begin{equation}\label{ineq:x_1_x_2_dist}
	\|x_1(t)-x_2(t)\|\leq Le^{LT}(t-s)\sup_{\theta\in [s,r']}W_1(m(\theta),m'(\theta)).
	\end{equation}
	Further,
	$$|z_1(t)-z_2(t)|\leq L\int_s^t\|x_1(\theta)-x_2(\theta)\|d\theta+L(t-s)\sup_{\theta\in [s,r']}W_1(m(\theta),m'(\theta)).  $$ Using (\ref{ineq:x_1_x_2_dist}), we get
	\begin{equation}\label{ineq:z_1_z_2_dist}
	|z_1(t)-z_2(t)|\leq L(LTe^{LT}+1)(t-s)\sup_{\theta\in [s,r']}W_1(m(\theta),m'(\theta)).
	\end{equation}	
	Since $\psi(x_2(r))+z_2(r)\leq B^{s,r}_{m(\cdot)}\psi(y)$, from (\ref{ineq:x_1_x_2_dist}) and (\ref{ineq:z_1_z_2_dist})   we conclude that
	\begin{equation*}
	\begin{split}
	B^{s,r'}_{m'(\cdot)}\psi'(y)=\psi'(x_1(r')&)+z_1(r')-z_1(s)
	\\ \leq \psi(x_1(r')&)+z_1(r')-z_1(s)+\|\psi-\psi'\|\\ \leq \psi(x_1(r))&+z_1(r)-z_1(s)+(K+1)R(r'-r)+\|\psi-\psi'\|
	\\\leq
	\psi(x_2(r))&+z_2(r)-z_2(s)+K\|x_1(r)-x_2(r)\|+|z_1(r)-z_2(r)|\\&+(K+1)R(r'-r)+\|\psi-\psi'\|\\ 
	\leq
	B^{s,r}_{m(\cdot)}\psi(y&) +L(Ke^{LT}+LTe^{LT}+1)(r-s)\sup_{\theta\in [s,r']}W_1(m(\theta),m'(\theta))\\&+(K+1)R(r'-r)+\|\psi-\psi'\|.
	\end{split}
	\end{equation*}
	Analogously, one can prove that
	\begin{equation*}
	\begin{split}
	B^{s,r}_{m(\cdot)}\psi(y)\leq
	B^{s,r'}_{m'(\cdot)}\psi'(y)+L(Ke^{LT}+LTe^{LT}+1)&(r-s)\sup_{\theta\in [s,r']}W_1(m(\theta),m'(\theta))\\&+(K+1)R(r'-r)+\|\psi-\psi'\|.
	\end{split}
	\end{equation*}	  These two inequalities yield the conclusion of the lemma.
\end{proof}

\section{Value multifunction}\label{sect:value_multifunction}

In this section we introduce the notion of the value multifunction that describes the dependence of the solution of the mean field game on the initial distribution and examine this dependence using the viability approach.

\begin{definition}\label{def:value_multifunction}
	We say that a upper semicontinuous function multifunction $\mathcal{V}:[0,T]\times\ptd\rightrightarrows C(\td)$ is a value multifunction of mean field game (\ref{sys:dynamics}), (\ref{sys:payoff}) if,  for any $t_0\in [0,T]$, $m_0\in\ptd$, and $\phi\in \mathcal{V}(t_0,m_0)$, there exists a  solution to mean field game (\ref{sys:dynamics}), (\ref{sys:payoff}) $(V,m(\cdot))$ such that
	\begin{equation}\label{property:value_function}
	V(t_0,\cdot)=\phi(\cdot),\ \ m(t_0)=m_0. 
	\end{equation}
\end{definition}
The assumption that the value function is upper semicontinuous is not restrictive. This is supported by the following. 
\begin{proposition}\label{prop:value_limit}
	If $\mathcal{V}:[0,T]\times\ptd\rightrightarrows C(\td)$ is such that, for any $t_0\in [0,T]$, $m_0\in\ptd$, $\phi\in\mathcal{V}(t_0,m_0)$, there exists a solution to mean field game (\ref{sys:dynamics}), (\ref{sys:payoff}) $(V,m(\cdot))$ satisfying (\ref{property:value_function}), then its closure $\mathcal{V}^\natural$ is a value multifunction. Here $\mathcal{V}^\natural$ is defined by the rule: $\phi$ belongs to $\mathcal{V}^\natural(t_0,m_0)$ iff there exist sequences $\{t_i\}\subset [0,T]$, $\{m_i\}\subset\ptd$ and $\{\phi_i\}\subset C(\td)$ such that $(t_i,m_i)\rightarrow (t_0,m_0)$, $\|\phi_i-\phi\|\rightarrow 0$ as $i\rightarrow\infty$.
\end{proposition} 
\begin{proof} First, notice that $\mathcal{V}^\natural$ is upper semicontinuous.
	
	Now, let $\phi\in\mathcal{V}^\natural(t_0,m_0)$. We shall prove that $\phi=V(t_0,\cdot)$, $m_0=m(t_0)$ for some solution to mean field game  (\ref{sys:dynamics}), (\ref{sys:payoff}). Let $\{t_i\}$, $\{m_i\}$, $\{\phi_i\}$ are such that $\phi_i\in\mathcal{V}(t_i,m_i)$ and $(t_i,m_i)\rightarrow (t_0,m_0)$, $\|\phi_i-\phi\|\rightarrow 0$. Since $\phi_i\in\mathcal{V}(t_i,m_i)$, one can find a pair $(V_i,\bar{m}_i(\cdot))$ that is a solution of mean field game satisfying the properties $V_i(t_i,\cdot)=\phi_i(\cdot)$, $m_i=\bar{m}_i(t_i)$. Hence, by Proposition \ref{prop:equivalence} there exists $\nu_i(\cdot)$ solving MFDI (\ref{system:mfdi}) such that, for any $s\in [t_i,T]$,  
	\begin{equation}\label{properties:V_m_nu_i}
	\bar{m}_i(s)=\mathrm{p}_\#\nu_i(s),\ \ V_i(s,\cdot)=B_{\bar{m}_i(\cdot)}^{s,T}\sigma(\cdot,\bar{m}_i(T)),\ \ [\sigma(\cdot,\bar{m}_i(T)),\nu_i(T)]\geq [V_i(s,\cdot),\nu(s)].
	\end{equation} Without loss of generality, one can assume that $\nu_i(t)$ and $V_i(t,\cdot)$ are defined for $t\in [t^*,T]$, where $t^*=\inf\{t_0,t_1,t_2,\ldots\}$. Moreover, by Proposition \ref{prop:shift} we additionally suppose that $\nu_i(t_i)=\widehat{m_i}$. Here $\widehat{m_i}$ are defined by (\ref{intro:lifting_m}).
	
	Using Proposition \ref{prop:limit}, get that there exist a subsequence $\{\nu_{i_k}\}$ and a solution to MFDI (\ref{system:mfdi}) $\nu(\cdot)$ such that $\sup_tW_1(\nu_{i_k}(t),\nu(t))\rightarrow 0$ and $\nu(t_0)=\widehat{m_0}$. Put $m(t)\triangleq \mathrm{p}_\#\nu(t)$. We have that \begin{equation}\label{limit:W_m_i}
	\sup_tW_1(\bar{m}_{i_k}(t),m(t))\rightarrow 0.
	\end{equation} Further, the functions $V_i$ are Lipschitz continuous with the constant depending only on the bounds of $f$  $g$ and Lipschitz constant of $\sigma$ w.r.t. $x$ (see  Lemma \ref{lm:lipschitz}). This implies that $\{V_i\}$ is also relatively compact. Without loss of generality, we can assume that $\{V_{i_k}\}$ converges to some function $V:[t_0,T]\times\td\rightarrow\mathbb{R}$. Passing to the limit in (\ref{properties:V_m_nu_i}), using (\ref{limit:W_m_i}), Lemma \ref{lm:B_continuity} and Proposition \ref{prop:equivalence}, we conclude that  $(V,m(\cdot))$ is a solution to mean field game (\ref{sys:dynamics}), (\ref{sys:payoff}) and satisfies $V(t_0,\cdot)=\phi(\cdot)$, $m(t_0)=m_0$.
\end{proof}
\begin{remark} 
	Note that the value multifunction is not defied in the unique way. It is natural to say that $\mathcal{W}:[0,T]\times\ptd\rightrightarrows C(\td)$ is the maximal value multifunction if it is a value multifunction and, for any value multifunction $\mathcal{V}$, $$\mathcal{V}(t,m)\subset \mathcal{W}(t,m), \ \ t\in [0,T], \ \ m\in\ptd. $$ The existence result for the maximal value function can be proved using the facts that the closure of union of value multifunctions is   also a value multifucntion (see Proposition \ref{prop:value_limit}). Furthermore, it follows from the existence theorem for the solution to the mean field game that $\mathcal{W}(t,m)\neq \varnothing$ for any $t\in [0,T]$ and $m\in\ptd$.
\end{remark}

\begin{remark}
	Let us describe the relation between the value multifunction and the master equation for mean field game.  For the case of the deterministic mean field game the master equation takes the form (see \cite[(II.4.41)]{Carmona_Delarue_II}):
	\begin{equation}\label{eq:master}
	\frac{\partial\varphi}{\partial t}+H(t,x,m,\nabla_x\varphi)+\int_{\td}\frac{\partial H(t,x,m,\nabla_x\varphi)}{\partial p}\partial_m\varphi(t,x,m)(y)m(dy)=0. 
	\end{equation} Here $\varphi$ is a function from $[0,T]\times\td\times\ptd$ with values in $\mathbb{R}$, $\partial \varphi/\partial t$, $\nabla_x\varphi$, $\nabla_m\varphi$ stand for its derivatives w.r.t to time, state and measure variable. The precise definition of the derivative w.r.t. measure can be found, for example, at \cite{Cardaliaguet_Delarue_Lasry_Lions_2015} or \cite[\S 5]{Carmona_Delarue_I}.
	The master equation is strictly connected with the value function of mean field game $\Gamma$ which assigns to a triple $(t_0,x_0,m_0)$ an expected reward of the representative player who starts at the time $t_0$, at the state $x_0$ under the condition that the distribution of all players is $m_0$. It is proved that under certain regularity conditions the classical solution to the master equation (\ref{eq:master}) is a value function (see \cite[Proposition 4.1]{Carmona_master}).  Conversely, if the coefficients of (\ref{eq:master}) are continuous, the value function $\Gamma$ is well-defined, unique and continuous with its derivative w.r.t. $x$, then it solves master equation (\ref{eq:master}) in the viscosity sense \cite[Proposition 4.20]{Carmona_Delarue_II}.  However, the existence of the classical solution is proved only for the special case of nondegenerate stochastic mean field game \cite{Cardaliaguet_Delarue_Lasry_Lions_2015} or on the short time interval \cite{Carmona_Delarue_II}, \cite{Gangbo_Swiech_2015}, \cite{Sergio_Mayorga_2018}. Furthermore, since there are examples of multiplicity of  solutions to the mean field game \cite{Bardi_Fischer_2017}, the regularity conditions required in \cite[Proposition 4.20]{Carmona_Delarue_II}  are not fulfilled in the general case. 
	
	On the other hand, the notion of the value multifunction is a natural extension of the notion of the value function. Indeed, if $\Gamma$ is a value function, then the multifunction $\mathcal{V}$ defined by the rule 
	\begin{equation}\label{intro:multifunction_by_function}
	\mathcal{V}(t_0,m_0)\triangleq \{\Gamma(t_0,\cdot,m_0)\}
	\end{equation} is a value multifunction. Conversely, if $\mathcal{V}$ is single-valued i.e. $\mathcal{V}(t_0,m_0)=\{\phi_{t_0,m_0}\}$, then the function $\Gamma(t_0,x_0,m_0)\triangleq \phi_{t_0,m_0}(x_0)$ is a value function. This implies that the classical solution of master equation~(\ref{eq:master}) gives a value multifunction. Simultaneously, if the value multifunction is single-valued and sufficiently smooth it provides the viscosity solution to master equation~(\ref{eq:master}).
\end{remark}

We look for  a sufficient condition for a multifunction $\mathcal{V}:[0,T]\times\ptd\rightrightarrows C(\td)$ to be a value multifunction. To this end we introduce the dynamical system on $\ptd\times C(\td)$.

\begin{definition}\label{def:flow} For each $s,r\in [0,T]$, $s\leq r$, define the multifunction $\Psi^{r,s}:\ptd\times C(\td)\rightrightarrows \ptd\times C(\td)$ by the rule: $(\mu,\psi)\in \Psi^{r,s}(m,\phi)$ if and only if  there exists a solution of MFDI (\ref{system:mfdi}) on $[s,r]$ $\nu(\cdot)$, satisfying the following properties  for $m(t)=\mathrm{p}_\#\nu(t)$:
	\begin{list}{($\Psi$\arabic{tmp})}{\usecounter{tmp}}
		\item\label{psi_cond:init_bound} $m(s)=m$, $m(r)=\mu$;
		\item\label{psi_cond:B} $\phi=B^{s,r}_{m(\cdot)}\psi$;
		\item\label{psi_cond:int_ineq} $[\psi,\nu(r)]\geq [\phi,\nu(s)]$.
	\end{list} 
\end{definition} 
We shall show that the family of set-valued mappings $\Psi=\{\Psi^{r,s}\}_{s\leq r}$ provides  a dynamics in the space $\ptd\times C(\td)$. First, notice that $\Psi^{s,s}(m,\phi)=\{(m,\phi)\}$.

\begin{proposition}\label{prop:semigroup} $\Psi$ satisfies the semigroup property i.e., for any $s_0,s_1,s_2\in [0,T]$, $s<r<\theta$,
	$$\Psi^{s_2,s_0}=\Psi^{s_2,s_1}\circ\Psi^{s_1,s_0}. $$
\end{proposition}

Here we define the composition of multivalued maps in the usual manner, i.e. if $(X,\rho_X)$ is a metric space, then  the composition of multivalued mappings $\Phi_1,\Phi_2:X\rightrightarrows X$ is 
$$(\Phi_2\circ\Phi_1)(x)\triangleq \bigcup_{y\in \Phi_1(x)}\Phi_2(y).$$ 
\begin{proof}[Proof of Proposition \ref{prop:semigroup}]
	First, let us prove that $\Psi^{s_2,s_1}\circ\Psi^{s_1,s_0}\subset \Psi^{s_2,s_0}$.  Pick $(m^0,\phi^0),(m^1,\phi^1),(m^2,\phi^2)\in\ptd\times C(\td)$ such that
	$(m^1,\phi^1)\in \Psi^{s_1,s_0}(m^0,\phi^0)$, $(m^2,\phi^2)\in \Psi^{s_2,s_1}(m^1,\phi^1)$. We are to prove that $(m^2,\phi^2)\in \Psi^{s_2,s_0}(m^0,\phi^0)$. For $i=1,2$, let $\nu_i(\cdot)\in\mathcal{N}_{s_{i-1},s_i}$, $m_i(\cdot)\in\mathcal{M}_{s_{i-1},s_{i}}$ be such that
	$m_i(t)=\mathrm{p}_\# \nu_i(t)$, $\nu_i(\cdot)$ solves MFDI~(\ref{system:mfdi}) on $[s_{i-1},s_i]$, $m_i(s_{i-1})=m^{i-1}$, $m_i(s_i)=m^i$, $\phi^{i-1}=B^{s_{i-1},s_{i}}_{m_i(\cdot)}\phi^{i}$, $[\phi^i,\nu_i(s_i)]\geq [\phi^{i-1},\nu_i(s_{i-1})]$. Due to Proposition \ref{prop:shift}, one may assume that $\nu_1(s_1)=\nu_2(s_1)$. Let 
	$$\nu(t)\triangleq\left\{\begin{array}{cc}
	\nu_1(t), & t\in [s_0,s_1],\\
	\nu_2(t), & t\in [s_1,s_2],
	\end{array}
	\right. $$
	$m(t)\triangleq \mathrm{p}_\#\nu(t)$. Notice that $m(t)$ coincides with $m_1(t)$ on $[s_0,s_1]$ and with $m_2(t)$ on $[s_1,s_2]$. By Proposition \ref{prop:concatination} we have that $\nu(\cdot)$ is a solution to (\ref{system:mfdi}) on $[s_0,s_2]$. Using the semigroup property for the family of operators $B$ (see (\ref{equality:B_semogroup})), we get
	$$\phi^0=B^{s_0,s_2}_{m(\cdot)}\phi^2. $$ Further, $$[\phi^2,\nu(s_2)]=[\phi^2,\nu_2(s_2)]\geq  [\phi^1,\nu_2(s_1)]=[\phi^1,\nu_1(s_1)]\geq [\phi^0,\nu_1(s_0)]=[\phi^0,\nu(s_0)].$$
	Thus, we obtain the inclusion $\Psi^{s_2,s_1}\circ\Psi^{s_1,s_0}\subset \Psi^{s_2,s_0}$.
	
	Now we turn to the opposite inclusion. Pick $(m^0,\phi^0)$, $(m^2,\phi^2)$ such that 
	$(m^2,\phi^2)\in\Psi^{s_2,s_0}(m^0,\phi^0)$. We shall prove that there exists $(m^1,\phi^1)$ such that
	$(m^1,\phi^1)\in\Psi^{s_1,s_0}(m^0,\phi^0)$ and $(m^2,\phi^2)\in\Psi^{s_2,s_1}(m^1,\phi^1)$. Let $\nu(\cdot)\in\mathcal{N}_{s_0,s_2}$, $m(\cdot)\in \mathcal{M}_{s_0,s_2}$ satisfy conditions ($\Psi$1)--($\Psi$3) for $s=s_0$, $r=s_2$, $m=m^0$, $\phi=\phi^0$, $\mu=m^2$, $\psi=\phi^2$. There exists $\chi\in \mathcal{P}^1(\mathcal{C}_{s_0,s_2})$  such that 
	$\nu(t)=\hat{e}_t{}_\#\chi $ and $\mathrm{supp}(\chi)\subset \mathrm{SOL}(s_2,s_0,m(\cdot))$. Notice that the function
	$\phi^0$ is a value  of the problem
	\begin{equation}\label{problem:max_phi_2}
	\text{maximize }\phi^2(x(s_2))+z(s_2)-z(s_0) 
	\end{equation}
	\begin{equation}\label{problem:constrains_phi_2}
	\text{subject to }(x(\cdot),z(\cdot))\in\mathrm{SOL}(s_2,s_0,m(\cdot)). 
	\end{equation} Additionally, one can rewrite ($\Psi$3) in the form
	$$\int_{\mathcal{C}_{s_0,s_2}}(\phi^2(x(s_2))+(z(s_2)-z(s_0))\chi(d(x(\cdot),z(\cdot)))\geq \int_{\mathcal{C}_{s_0,s_2}} \phi^0(x(s_0))\chi(d(x(\cdot),z(\cdot))). $$ Using inclusion $\mathrm{supp}(\chi)\subset \mathrm{SOL}(s_2,s_0,m(\cdot))$ we conclude that $\chi$ is concentrated on the set of optimal motions to  problem (\ref{problem:max_phi_2}), (\ref{problem:constrains_phi_2}).
	
	Now set $m^1=m(s_1)$, $\phi^1=B^{s_1,s_2}_{m(\cdot)}\phi^2$. Further, for $t\in [s_{i-1},s_{i}]$, put  $\nu_i(t)$ be equal to $\nu(t)$, $m_i(t)\triangleq \mathrm{p}_\#\nu_i(t)$. The dynamic programming gives that
	$$\phi^{i-1}=B^{s_{i-1},s_i}_{m_i(\cdot)}\phi^{i}, \ \ i=1,2. $$ It remains to prove that \begin{equation}\label{ineq:prop_Psi_phi}
	[\phi^i,\nu_i(s_i)]\geq [\phi^{i-1},\nu_i(s_{i-1})].
	\end{equation} Let $\chi_i$ be a projection of $\chi$ on $\mathrm{C}_{s_{i-1},s_i}$. We have that $\nu_i(t)=\hat{e}_t{}_\#\chi_i$, $\mathrm{supp}(\chi_i)\subset \mathrm{SOL}(s_i,s_{i-1},m_i(\cdot))$. Moreover, since $\chi$ is concentrated on optimal motions to~(\ref{problem:max_phi_2}),~(\ref{problem:constrains_phi_2}), we have that each motion $(x(\cdot),z(\cdot))\in\mathrm{supp}(\chi_i)$ provides the solution to the problem 
	\begin{equation*}\label{problem:max_phi_i}
	\text{maximize }\phi^i(x(s_i))+z(s_i)-z(s_{i-1}) 
	\end{equation*}
	\begin{equation*}\label{problem:constrains_phi_i}
	\text{subject to }(x(\cdot),z(\cdot))\in\mathrm{SOL}(s_i,s_{i-1},m_i(\cdot)). 
	\end{equation*} Integrating this and taking into account the property $\mathrm{supp}(\chi_i)\subset \mathrm{SOL}(s_i,s_{i-1},m_i(\cdot))$, we get
	\begin{equation*}
	\int_{\mathcal{C}_{s_{i-1},s_i}}(\phi^i(x(s_i))+z(s_i))\chi_i(d(x(\cdot),z(\cdot)))\geq \int_{\mathcal{C}_{s_{i-1},s_i}} (\phi^{i-1}(x(s_{i-1}))+z(s_{i-1}))\chi_i(d(x(\cdot),z(\cdot))).
	\end{equation*}
	This implies (\ref{ineq:prop_Psi_phi}).
\end{proof}

Recall (see the Introduction) that we call the dynamics generated by the family $\{\Psi^{r,s}\}_{s\leq r}$ the mean field game dynamics.

\begin{definition}\label{def:dpp}
	We say that a upper semicontinuous   multifunction $\mathcal{V}:[0,T]\times \ptd\rightrightarrows C(\mathbb{T}^d)$ is viable with respect to the mean field game dynamics if, for any $s,r\in [0,T]$, $s\leq r$, $m\in\ptd$, $\phi\in \mathcal{V}(s,m)$, there exist $\mu\in \td$ and $\psi\in C(\td)$ such that 
	\begin{itemize}
		\item $(\mu,\psi)\in \Psi^{r,s}(m,\phi)$;
		\item $\psi\in \mathcal{V}(r,\mu)$.
	\end{itemize}
\end{definition}

\begin{remark}\label{remark:viability_concept} It is more accurate to say that the graph of the multifunction $\mathcal{V}$ is viable with respect to the mean field game dynamics. However, for the sake of shortness we will say that the multifunction $\mathcal{V}$ itself is viable.
\end{remark}

\begin{remark}\label{remark:max_viability}
	One can prove that the maximal value function is viable with respect to the mean field game dynamics.
\end{remark}

The link between viability property and value function is given in the following statement.

\begin{theorem}\label{th:DPP}
	Assume that a upper semicontinuous multifunction $\mathcal{V}:\ir\rightrightarrows C(\mathbb{T}^d)$  is viable with respect to the mean field game dynamics  and $\mathcal{V}(T,m)=\{\sigma(\cdot,m)\}$. Then $\mathcal{V}$ is a value multifunction.
\end{theorem}


\begin{proof}
	We shall prove that if $\mathcal{V}$ is viable with respect to the mean field game dynamics, then, for any $t_*\in [0,T]$, $m_*\in\ptd$ and $\phi_*\in\mathcal{V}(t_*,m_*)$, there exists a  solution to (\ref{sys:dynamics}), (\ref{sys:payoff}) $(V,m(\cdot))$ such that $V(t_*,\cdot)=\phi_*(\cdot)$ and $m(t_*)=m_*$.
	
	Let $N$ be a natural number, and let $t^i_N\triangleq t_*+(T-t_*)/N$. By assumption we have that there exists a sequence of pairs $\{(\mu_N^i,\phi_N^i)\}_{i=0}^N\subset \mathcal{P}^1(\td)\times C(\td)$ satisfying the following conditions:
	\begin{itemize}
		\item $\phi^i_N\in\mathcal{V}(t_N^i,\mu_N^i)$, $i=0,1,\ldots,N$;
		\item $\phi^0_N=\phi_*$, $\mu^0_N=m_*$;
		\item $(\mu_N^i,\phi_N^{i})\in \Psi^{t_N^i,t_N^{i-1}}(\mu_N^{i-1},\phi_N^{i-1})$, $i=1,\ldots,N$;
		\item $\phi_N^N(\cdot)=\sigma(\cdot,\mu_N^N)$.
	\end{itemize} 
	Since $(\mu_N^i,\psi_N^i)\in\Psi^{t_N^i,t_N^{i-1}}(\mu_N^{i-1},\psi_N^{i-1})$, there exist $\nu^i_N(\cdot)\in\mathcal{N}_{t_N^{i-1},t_N^i}$ and $m^i_N(\cdot)\in\mathcal{M}_{t_N^{i-1},t_N^i}$  such that
	\begin{itemize}
		\item $\nu_N^i(\cdot)$ is a solution of (\ref{system:mfdi}) on $[t^{i-1}_N,t^i_N]$;
		\item $m_N^i(t)=\mathrm{p}_\#\nu_N^i(t)$ when $t\in [t^{i-1}_N,t^i_N]$;
		\item $m_N^i(t_N^{i-1})=\mu_N^{i-1}$, $m_N^i(t_N^i)=\mu_N^i$;
		\item $\phi_N^i=B^{t_N^i,t_N^{i-1}}_{m_N^i(\cdot)}\phi_N^{i-1}$;
		\item $[\phi^i_N,\nu_N^i(t_N^i)]\geq [\phi^{i-1}_N,\nu_N^i(t_N^{i-1})]$.
	\end{itemize}
	By Proposition \ref{prop:shift} we can assume without loss of generality that $\nu_N^{i-1}(t_N^{i-1})=\nu_N^{i}(t_N^{i-1})$. Define $\nu_N(\cdot)\in \mathcal{N}_{t_*,T}$ by the rule: if $t\in [t_N^{i-1},t_N^i]$, then
	$$\nu_N(t)\triangleq \nu_N^i(t). $$  Further, let $m_N(t)\triangleq \mathrm{p}_\#\nu_N(t)$.  By Proposition \ref{prop:concatination} we have that $\nu_N(\cdot)$ solves (\ref{system:mfdi}). Moreover,  $m_N(t_*)=m_*$.  Proposition \ref{prop:limit} implies that there exist a flow of probabilities $\nu(\cdot)\in \mathcal{N}_{t_*,T}$ and a sequence $\{N_k\}_{k=1}^\infty$ such that
	$\nu(\cdot)$ is a solution of (\ref{system:mfdi}), and $$\lim_{k\rightarrow\infty}\sup_{t\in [t_*,T]}W_1(\nu_{N_k}(t),\nu(t))=0. $$ Put \begin{equation}\label{intro:m_limit}
	m(t)\triangleq \mathrm{p}_\#\nu(t).
	\end{equation} 
	
	For $s\in [t_*,T]$, set 
	\begin{equation}\label{intro:V_N}
	V_N(s,\cdot)=B^{s,T}_{m_N(\cdot)}\sigma(\cdot,m_N(T)).
	\end{equation}
	By Lemma \ref{lm:lipschitz} the functions $ V_N$  are $C_1$-Lipschitz continuous functions, where the constant $C_1$ does not depend on $N$. Hence, without loss of generality, one can assume that the sequence $\{V_{N_k}\}$ converges to some function $V\in C([t_*,T]\times\td)$. 
	
	Let us prove that $(V,m(\cdot))$ is a solution of mean field game (\ref{sys:dynamics}), (\ref{sys:payoff}). To this end we check the conditions of Proposition \ref{prop:equivalence}. 
	By construction we have that $m(t)=\mathrm{p}_\#\nu(t)$, where $\nu(\cdot)$ solves MFDI (\ref{system:mfdi}).
	Notice that $W_1(m_{N_k}(t),m(t))\leq W_1(\nu_{N_k}(t),\nu(t))$. Thus,
	\begin{equation}\label{converge:m_N_k_m}
	\lim_{k\rightarrow\infty}\sup_{t\in [t_*,T]}W_1(m_{N_k}(t),m(t))=0.
	\end{equation} It follows from (\ref{intro:V_N}) that, for any $k$,
	\begin{equation*}
	\|V(s,\cdot)-B_{m(\cdot)}^{s,T}\sigma(\cdot,m(T))\|
	\leq
	\|V(s,\cdot)-V_{N_k}(s,\cdot)\|  +\|B_{m_{N_k}(\cdot)}^{s,T}\sigma(\cdot,m_{N_k}(T))-B_{m(\cdot)}^{s,T}\sigma(\cdot,m(T))\|.
	\end{equation*}
	Using Lemma \ref{lm:B_continuity}, we estimate the right-hand side of this inequality and get
	\begin{equation*}\begin{split}\|V(s,\cdot)-B_{m(\cdot)}^{s,T}&\sigma(\cdot,m(T))\|\\ \leq \|V(s,\cdot&)-V_{N_k}(s,\cdot)\|+ \|\sigma(\cdot,m(T))-\sigma(\cdot,m_{N_k}(T))\|\\
	&+L(\varkappa e^{LT}+LTe^{LT}+1)(T-s)\sup_{t\in [t_*,T]}W_1(m_{N_k}(t),m(t)).
	\end{split}
	\end{equation*}
	Passing to the limit when $k\rightarrow 0$ and using (\ref{converge:m_N_k_m}), we conclude that
	\begin{equation}\label{equal:V_B}
	\|V(s,\cdot)-B_{m(\cdot)}^{s,T}\sigma(\cdot,m(T))\|=0.
	\end{equation} 
	It remains to prove that 
	\begin{equation}\label{ineq:action_V_nu}
	[\sigma(\cdot,m(T)),\nu(T)]\geq [V(s,\cdot),\nu(s)].
	\end{equation} To this end notice that 
	$\nu_N(t)=\nu_N^i(t)$, for $t\in [t_N^{i-1},t_N^i]$, $V_N(t_N^i,\cdot)=\phi_N^i(\cdot)$ and $V_N(T,\cdot)=\sigma(\cdot,m_N(T))$. This and the inequalities $[\phi_N^i,\nu_N^i(t_N^i)]\geq [\phi_N^{i-1},\nu_N^i(t_N^{i-1})]$ yield that
	\begin{equation*}
	[\sigma(\cdot,m_N(T)),\nu_N(T)]\geq [V_N(t_N^{i-1},\cdot),\nu_N(t_N^{i-1})],\ \ i=1,\ldots,N.
	\end{equation*}
	Further, let $i$ be such that $s\in [t_N^{i-1},t_N^i]$. We have that
	$\|V_N(s,\cdot)-V_N(t_N^{i-1},\cdot)\|\leq C_1/N$. Since $\nu_N(\cdot)$ solves MFDI (\ref{system:mfdi}),
	$W_1(\nu_N(s),\nu_N(t_N^{i-1}))\leq 2R/N$. Therefore, by (\ref{estima:action_continuous})
	\begin{equation}\label{ineq:sigma_n_N_v_N}
	[\sigma(\cdot,m_N(T)),\nu_N(T)]\geq [V_N(s,\cdot),\nu_N(s)]-C_1(1+2R)/N.
	\end{equation}
	Further,  the functions $x\mapsto \sigma(x,m(T))$, $x\mapsto V(s,x)$ are Lipschitz continuous with the constants $\varkappa$ and $C_1$ respectively. Hence, using (\ref{ineq:sigma_n_N_v_N}) for $N=N_k$ and (\ref{estima:action_continuous}), we conclude
	\begin{equation*}
	\begin{split}
	[\sigma(\cdot,m(T)),\nu(T)]\geq [V(s,\cdot),&\nu(s)]-C_1(1+2R)/N_k-\|\sigma(\cdot,m(T))-\sigma(\cdot,m_{N_k}(T))\|\\&- \|V(s,\cdot)-V_{N_k}(s,\cdot)\|-\varkappa W_1(\nu(T),\nu_{N_k}(T))-C_1W_1(\nu(s),\nu_{N_k}(s)).
	\end{split}
	\end{equation*}
	Passing to the limit, we get (\ref{ineq:action_V_nu}).
	
	The fact that $\nu(\cdot)$ is a solution to (\ref{system:mfdi}), equalities (\ref{intro:m_limit}), (\ref{equal:V_B})and inequality (\ref{ineq:action_V_nu}) imply that $(V,m(\cdot))$ is a solution of mean field game (\ref{sys:dynamics}), (\ref{sys:payoff}) (see Proposition \ref{prop:equivalence}). Furthermore, by construction $m(t_*)=m_*$ and $V(t_*,m_*)=\phi_*$. Thus, $\mathcal{V}$ is a value multifunction. 
\end{proof}

\section{Statement of the viability theorem}\label{sect:viability_theorem}

In this section we formulate the infinitesimal form of the viability property for mean field game dynamics.

To this end we use the probabilities on tangent bundles to $\td\times\{0\}$ and $\tdr$. Let $m\in\ptd$. Denote by $\mathcal{L}(m)$ the set of probabilities $\beta\in\mathcal{P}^1(\td\times\rdp)$ with the marginal on $\td$ equal to $m$. Analogously, if $\nu\in\mathcal{P}^1(\tdr)$, then $\mathcal{L}_*(\nu)$ stands for the set of probabilities $\gamma\in\mathcal{P}^1(\tdr\times \rdp)$ with the marginal on $\tdr$ equal to $\nu$. If $c>0$, then put 
\begin{equation}\label{intro:L_c}
\mathcal{L}^c(m)\triangleq \{\beta\in\mathcal{L}(m):\mathrm{supp}\beta\subset \td\times \mathbb{B}_c\times [-c,c]\},
\end{equation} 
\begin{equation}\label{intro:L_star_c}
\mathcal{L}^c_*(m)\triangleq \{\gamma\in\mathcal{L}_*(m):\mathrm{supp}\gamma\subset \tdr\times \mathbb{B}_c\times [-c,c]\}.
\end{equation} 
Here $\mathbb{B}_c$ stands for the closed ball in $\rd$ of  the radius $c$ centered at the origin. 

If $m\in\ptd$, $s,r\in [0,T], $ $s\leq r$, then denote by $A^{s,r}_m$ the operator  on $C(\mathbb{T}^d)$ acting by the rule
\begin{equation}\label{intro:A_Bellman}
(A^{s,r}_{m}\phi)(x)\triangleq \sup\{\phi(x+(r-s) a)+(r-s) b:(a,b)\in F(s,x,m)\}. 
\end{equation} Here $F$ is defined by (\ref{intro:F}). For $t\geq 0$, let the operator $\Theta^\tau:\td\times\rdp\rightarrow\tdr $ be given by
\begin{equation}\label{intro:Theta}
\Theta^\tau(x,a,b)\triangleq (x+\tau a,\tau b). 
\end{equation}

Let $\mathcal{V}:[0,T]\times\ptd\rightrightarrows C(\mathbb{T}^d)$ be upper semicontinuous, $t\in [0,T]$, $m\in \ptd$, $\phi\in C(\mathbb{T}^d)$. Now, we turn to the definition of set-valued derivative of the multifunction $\mathcal{V}$ at $t,m,\phi$ by virtue of $F$ under constraints determined by constant~$c$. We denote this derivative by  $\mathcal{D}^c_F\mathcal{V}(t,m,\phi)$. 

\begin{definition}\label{def:tangent}
	A probability $\beta\in\mathcal{L}^c(m)$ belongs to $\mathcal{D}_{F}^c\mathcal{V}(t,m,\phi)$ if    there exist sequences $\{\tau_n\}_{n=1}^\infty\subset (0,+\infty)$, $\{\beta_n\}_{n=1}^\infty\subset\mathcal{L}^c(m)$ and $\{\phi_n\}_{n=1}^\infty\subset\ C(\mathbb{T}^d)$ satisfying the following properties for $\nu_n\triangleq \Theta^{\tau_n}{}_\#\beta_n$ and $m_n\triangleq \mathrm{p}_\#\nu_n$:
	\begin{enumerate}
		\item $\tau_n,W_1(\beta,\beta_n)\rightarrow 0$ as $n\rightarrow\infty$;
		\item $\phi_n\in\mathcal{V}(t+\tau_n,m_n)$;
		\item\label{def_tangent:cond_proximate} $$\lim_{n\rightarrow\infty}\frac{\|A^{t,t+\tau_n}_m\phi_n-\phi\|}{\tau_n}=0; $$
		\item\label{def_tangent:cond_optima} 
		$$\lim_{n\rightarrow\infty}\frac{[\phi_n,\nu_n]-[\phi,\widehat{m}]}{\tau_n}\geq 0;$$
		\item $$\int_{\td\times\rdp}\mathrm{dist}(v;F(t,x,m))\beta(d(x,v))=0. $$		
	\end{enumerate}
\end{definition}

For $M,C>0$, let $\mathrm{BL}_{M,C}$ denote the set of  functions $\phi\in \mathrm{Lip}_C(\td)$ such that
$$\|\phi\|\leq M.  $$

\begin{theorem}\label{th:viability}
	Assume that the upper semicontinuous  multifunction $\mathcal{V}:[0,T]\times\ptd\rightrightarrows C(\td)$ has nonempty values and there exist constants $M$ and $C$ such that, for any $t\in [0,T]$, $m\in\ptd$, 
	$$\mathcal{V}(t,m)\subset \mathrm{BL}_{M,C}(\td).$$
	Then, $\mathcal{V}$ is viable with respect to the mean field game dynamics if and only if, there exists a constant $c>0$ such that, for any $t\in [0,T]$, $m\in\ptd$, $\phi\in\mathcal{V}(t,m)$,
	$$\mathcal{D}^c_F(t,m,\phi)\neq\varnothing. $$
\end{theorem} Here $\widehat{m}$ is defined by (\ref{intro:lifting_m}).

Theorems \ref{th:DPP}, \ref{th:viability} immediately implies the following.
\begin{corollary}
	Let the upper semicontinuous  multifunction $\mathcal{V}:[0,T]\times\ptd\rightrightarrows C(\td)$ have nonempty values. Assume that, for any $t\in [0,T]$, $m\in\ptd$, $\phi\in\mathcal{V}(t,m)$,
	\begin{itemize}
		\item $\mathcal{V}(t,m)\subset \mathrm{BL}_{M,C}(\td)$
		where the constants $M$ and $C$ do not dependent on $t$ and $m$;
		\item $\mathcal{V}(T,m)=\{\sigma(\cdot,m)\}$;
		\item $\mathcal{D}^c_F(t,m,\phi)\neq\varnothing, $ where the constant $c$ does not depend on $t$, $m$ and $\phi$.
	\end{itemize} 
	Then $\mathcal{V}$ is a value multifunction of mean field game (\ref{sys:dynamics}), (\ref{sys:payoff}). 
\end{corollary}

\section{Properties of the ``frozen'' dynamics}\label{sect:auxiliary}
In this Section we present  auxiliary statements those are used in the proof of Theorem \ref{th:viability}.

For each $\tau\geq 0$, let $\Xi^\tau$ be an operator from $\tdr\times\rdp$ to $\tdr$ defined by the rule
\begin{equation}\label{intro:Xi}
\Xi^{\tau}(w,v)\triangleq w+\tau v. 
\end{equation}

\begin{lemma}\label{lm:Theta_pi}
	If $\nu,\nu'\in\ptd$,  $\gamma\in \mathcal{L}_*(\nu)$, $\tau,\tau'\geq 0$, $\pi^0\in\Pi(\nu',\nu)$ is an optimal plan between $\nu'$ and $\nu$,
	then
	\begin{enumerate}
		\item $$W_1(\Xi^{\tau}{}_\#\gamma,\Xi^{\tau'}{}_\#\gamma)\leq |\tau-\tau'|\int_{\tdr\times\rdp} \|v\|\gamma(d(w,v));$$
		\item
		$$W_1(\Xi^\tau{}_\#\gamma,\Xi^\tau{}_\#(\pi^0*\gamma))\leq W_1(\nu,\nu'). $$
	\end{enumerate}
\end{lemma}
\begin{proof}
	First, we have that
	\begin{equation*}
	\begin{split}
	W_1(\Xi^{\tau}&{}_\#\gamma,\Xi^{\tau'}{}_\#\gamma)\\&= \sup_{\varphi\in\mathrm{Lip}_1(\tdr)}\left\{\int_{\tdr\times\rdp}[\varphi(w+\tau v)-\varphi(w+\tau' v)]\gamma(d(w,v))\right\} \\
	&\leq |\tau-\tau'|\int_{\tdr\times\rdp} \|v\|\gamma(d(w,v)).
	\end{split}
	\end{equation*} This proves the first statement of the Lemma.
	
	Further, we have that
	\begin{equation*}
	\begin{split}
	W_1(&\Xi^\tau{}_\#\gamma,\Xi^\tau{}_\#(\pi^0*\gamma))\\
	&=\sup_{\varphi\in\mathrm{Lip}_1(\tdr)} \left\{\int_{(\tdr)^2}\int_{\rdp}
	[\varphi(w'+\tau v)-\varphi(w+\tau v)]\gamma(dv|w)\pi^0(d(w',w))\right\} \\
	&\leq \int_{(\tdr)^2}\|w'-w\|\pi^0(d(w',w)).
	\end{split}
	\end{equation*} This inequality yields the second statement of the Lemma.
\end{proof}

The following statements are concerned with the continuity of the operator $A$.

\begin{lemma}\label{lm:A_lipshitz}
	If $m\in\ptd$, $s,r\in [0,T]$, $s<r$, $K>0$, $\psi\in \mathrm{Lip}_{K}(\td)$, then $A^{s,r}_m\psi\in\mathrm{Lip}_{[(K +1)e^{L(r-s)}-1]}(\td)$. 
\end{lemma} The lemma is proved in the same way as the first statement of Lemma \ref{lm:lipschitz}.

\begin{lemma}\label{lm:A_continuity}
	If $m,m'\in\ptd$, $K>0$, $\psi,\psi'\in \mathrm{Lip}_K(\td)$, $s,s',r\in [0,T]$, $s<r$, $s'<r$, then
	\begin{equation*}\begin{split}\|A^{s,r}_{m}\psi&-A^{s',r}_{m'}\psi'\|\\ &\leq \|\psi-\psi'\|+ (K+1)[L(r-\hat{s})W_1(m,m')+R|s-s'|+\alpha(s-s')\cdot(r-\hat{s})],\end{split}\end{equation*}
	where $\hat{s}\triangleq\inf\{s,s'\}$.
\end{lemma}
\begin{proof}
	First, recall that
	\begin{equation}\label{equal:F_probabilities}
	F(t,x,m)=\left\{ \int_U (f(t,x,m,u),g(t,x,m,u))\xi(du):\xi\in\mathcal{P}(U) \right\}.
	\end{equation} 
	Without loss of generality we can assume that $s\leq s'$. Thus, we have that
	\begin{equation*}\begin{split}
	(A^{s,r}_{m}\psi)(x)\hspace{17pt}&{}\\=\sup_{\xi\in\mathcal{P}(U)}&\int_U (\psi(x+(r-s) f(s,x,m,u))+(r-s)g(s,x,m,u))\xi(du)\\ \leq
	\sup_{\xi\in\mathcal{P}(U)}&\int_U (\psi(x+(r-s') f(s',x,m',u))+(r-s')g(s',x,m',u))\xi(du))\\+(K+&1)[L(r-s)W_1(m,m')+R|s-s'|+\alpha(s-s')\cdot(r-s)]\\ \leq \sup_{\xi\in\mathcal{P}(U)}&\int_U (\psi'(x+(r-s') f(s',x,m',u))+(r-s')g(s',x,m',u))\xi(du)\\ +\|\psi-&\psi'\| +(K+1)[L(r-s)W_1(m,m')+R|s-s'|+\alpha(s-s')\cdot(r-s)]\\ = (A^{s',r}_{m'}\psi')&(x)+\|\psi-\psi'\|\\+(K+&1)[L(r-s)W_1(m,m')+R|s-s'|+\alpha(s-s')\cdot(r-s)].
	\end{split}
	\end{equation*} This proves the Lemma.
\end{proof}

The following lemma provides the comparison of the original and the ``frozen'' dynamics.

\begin{lemma}\label{lm:bellman_close}
	Let $m(\cdot)$ be a flow of probabilities on $[s,r]$, $m_*\in\ptd$, and let $\psi,\psi'\in \mathrm{Lip}_{K}(\mathbb{T}^d)$. Assume that, for all $t\in [s,r]$,  $W_1(m(t),m_*)\leq \delta_1$, then
	$$\|A^{s,r}_{m_*}\psi-B_{m(\cdot)}^{s,r}\psi'\|\leq (K+1)(\alpha(r-s)+LR(r-s)+L\delta_1)(r-s)+\|\psi-\psi'\|. $$
\end{lemma}
\begin{proof}
	First, by Lemma \ref{lm:A_continuity} we conclude that
	\begin{equation}\label{lmbc_proof:A_estima}
	\|A_{m_*}^{s,r}\psi-A_{m_*}^{s,r}\psi'\|\leq\|\psi-\psi'\|.
	\end{equation}
	
	Let $y\in\mathbb{T}^d$ and let $(a^*,b^*)\in F(s,y,m_*)$ be such that
	$$(A^{s,r}_{m_*}\psi')(y)=\psi'(y+(r-s)a^*)+(r-s)b^*. $$
	Using (\ref{equal:F_probabilities}), we obtain that there exists a probability $\xi\in\mathcal{P}(U)$ such that
	$$ (a^*,b^*)=\int_{U}(f(t,x,m_*,u),g(t,x,m_*,u))\xi(du).$$
	Consider the motion $(x(\cdot),z(\cdot))$ satisfying
	$$\dot{x}(t)=\int_U f(t,x(t),m(t),u)\xi(du),\ \ \dot{z}(t)=\int_U g(t,x(t),m(t),u)\xi(du),$$ and $x(s)=y,$ $z(s)=0$.
	
	We have that $\|x(t)-y\|\leq R(t-s)$. Thus, for every $t\in [s,r]$,
	$$\|\dot{x}(t)-a^*\|\leq \alpha(t-s)+LR(t-s)+L\delta_1,$$
	$$ |\dot{z}(t)-b^*|\leq \alpha(t-s)+LR(t-s)+L\delta_1. $$
	Therefore,
	$$\|x(r)-y-(r-s)a^*\|\leq (\alpha(r-s)+LR(r-s)+L\delta_1)(r-s), $$
	$$|z(r)-(r-s)b^*\|\leq (\alpha(r-s)+LR(r-s)+L\delta_1)(r-s). $$
	
	This yields the inequality
	\begin{equation}\label{ineq:A_B:proof_lm_A_B}
	\begin{split}
	A^{s,r}_{m_*}\psi'&=\psi'(y+(r-s)a^*)+(r-s)b^* \\
	&\leq \psi'(x(r))+z(r)+(K+1)(\alpha(r-s)+LR(r-s)+L\delta_1)(r-s)\\
	&\leq B^{s,r}_{m(\cdot)}\psi'+(K+1)(\alpha(r-s)+LR(r-s)+L\delta_1)(r-s).
	\end{split}
	\end{equation}
	The opposite inequality
	$$B^{s,r}_{m(\cdot)}\psi'\leq A^{s,r}_{m_*}\psi'+(K+1)(\alpha(r-s)+LR(r-s)+L\delta_1)(r-s) $$ is proved in the same way. This, (\ref{lmbc_proof:A_estima}) and (\ref{ineq:A_B:proof_lm_A_B}) imply the conclusion of the Lemma.
\end{proof}

\section{Proof of the viability theorem. Sufficiency}\label{sect:proof_sufficiency}
In this section we assume that the upper semicontinuous multifunction $\mathcal{V}:[0,T]\times \ptd\rightrightarrows C(\td)$ has nonempty values and, for any $t\in [0,T]$, $m\in \ptd$, $\phi\in\mathcal{V}(t,m)$, the following properties hold true $$\phi \in\mathrm{BL}_{M,C}(\td), \ \ \mathcal{D}^c_F\mathcal{V}(t,m,\phi)\neq \varnothing.$$ Here $c$, $M$ and $C$ are constants that do not dependent on $s$, $m$ and $\phi$. 

We shall prove that under this assumption $\mathcal{V}$ is viable with respect to the mean field game dynamics.
The proof is a modification of the proof of the classic viability theorem presented in~\cite{Aubin}.

Denote by $\mathcal{Z}$ the set of triples $(t,\nu,\phi)\in [0,T]\times\mathcal{P}^1(\tdr)\times \mathrm{BL}_{M,C}(\td)$ such that $\mathrm{supp}(\nu)\subset \td\times [-ct,ct]$ and $\phi\in\mathcal{V}(t,\mathrm{p}_\#\nu)$. Further, for $n\in\mathbb{N}$, set $\mathcal{Z}_n\triangleq \mathcal{Z}\cap ([0,T-1/n]\times\mathcal{P}^1(\tdr)\times \mathrm{BL}_{M,C}(\td))$. Clearly, the sets $\mathcal{Z}$ and $\mathcal{Z}_n$ are compacts.

\begin{lemma}\label{lm:step}  
	There exists a number $\theta_n\in (0,1/n)$ such that, for any $(t,\nu,\phi)\in\mathcal{Z}_n$, one can find $(t^+,\nu^+,\phi^+)\in \mathcal{Z}$ and $\gamma\in \mathcal{L}_*^c(\nu)$ such that the following properties hold true for $m=\mathrm{p}_\#\nu$:
	\begin{enumerate}
		\item $W_1(\Xi^{t^+-t}{}_\#\gamma,\nu^+)<(t^+-t)/n$;
		\item $\|A_{m}^{t,t^+}\phi^+-\phi\|<(t^+-t)/n$;
		\item $[\phi^+,\nu^+]> [\phi,\nu]-(t^+-t)/n$;
		\item $$\int_{\tdr\times\rdp}\mathrm{dist}\left(v,\widehat{F}(t,w,\nu)\right)\gamma(d(z,v))< 1/n. $$
	\end{enumerate}
	Here $\widehat{F}$ defined by (\ref{intro:hat_F}).
\end{lemma}
\begin{proof}
	Let $(s,\eta,\psi)\in\mathcal{Z}_n$. Denote $\mu_{\eta}\triangleq \mathrm{p}_\#\eta$. Since $\mathcal{D}^c_F(s,\mu_\eta,\psi)\neq\varnothing$ and the mapping $\beta\mapsto\int_{\tdr\times\rdp}\mathrm{dist}(v,F(t,x,m))\beta(d(x,v))$ is continuous, there exist $\vartheta_{s,\eta,\psi}\in (0,1/n)$, $\omega_{s,\eta,\psi}\in\mathcal{L}^c(\mu_\eta)$, $\bar{\eta}^+_{s,\eta,\psi}\in\tdr$ and $\psi^+_{s,\eta,\psi}\in \mathrm{Lip}_C(\td)$ such that
	\begin{itemize}
		\item $\bar{\eta}^+_{s,\eta,\psi}=\Theta^{2\vartheta_{s,\eta,\psi}}{}_\#\omega_{s,\eta,\psi}$;
		\item $\psi^+_{s,\eta,\psi}\in\mathcal{V}(s+2\vartheta_{s,\eta,\psi},\mathrm{p}_\#\bar{\eta}^+_{s,\eta,\psi})$;
		\item $|A^{s,s+2\vartheta_{s,\eta,\psi}}_{\mu_\eta}\psi^+_{s,\eta,\psi}-\psi|< \vartheta_{s,\eta,\psi}/n$;
		\item $[\psi^+_{s,\eta,\psi},\bar{\eta}^+_{s,\eta,\psi}]> [\psi,\widehat{\mu_\eta}]-\vartheta_{s,\eta,\psi}/n$;
		\item $\int_{\td\times\rdp}\mathrm{dist}(v,F(t,x,m))\omega_{s,\eta,\psi}(d(x,v))< 1/(2n).$
	\end{itemize}

	Now define the probability $\zeta_{s,\eta,\psi}\in \mathcal{P}(\tdr\times\rdp)$ in the following way.
	If $\varphi\in C_b(\tdr\times\rdp)$,
	\begin{equation*}
	\label{intro:zeta_shigt}
	\begin{split}\int_{\tdr\times\rdp}&\varphi(x,z,a,b)\zeta_{s,\eta,\psi}(d(x,z,a,b))\\&\triangleq\int_{\tdr}\int_{\mathbb{R}} \int_{\rdp}\varphi(x,z,a,b)\omega_{s,\eta,\psi}(d(a,b)|x)\eta(dz|x)\mu_\eta(dx).
	\end{split}
	\end{equation*} 
	Clearly,  the projection of $\zeta_{s,\eta,\psi}$ on $\td\times\rdp$ is $\omega_{s,\eta,\psi}$, when its projection on $\tdr$ is $\eta$. Thus,
	$\zeta_{s,\eta,\psi}\in\mathcal{L}_*^c(\eta)$.  Let $\eta^+_{s,\eta,\psi}\triangleq \Xi^{2\vartheta_{s,\eta,\psi}}{}_\#\zeta_{s,\eta,\psi}$. We have that 
	$\mathrm{p}_\#\eta_{s,\eta,\psi}=\mathrm{p}_\#\bar{\eta}_{s,\eta,\psi}$ and 
	\begin{equation*}
	\begin{split}
	[\psi^+_{s,\eta,\psi},&\eta^+_{s,\eta,\psi}]=\int_{\tdr\times\rdp}(\psi^+_{s,\eta,\psi}(x)+z) \eta^+_{s,\eta,\psi}(d(x,z))\\&=
	\int_{\td}\int_{\mathbb{R}}\int_{\rdp}(\psi^+_{s,\eta,\psi}(x+2\vartheta_{s,\eta,\psi}a)+z+2\vartheta_{s,\eta,\psi}b) \omega_{s,\eta,\psi}(d(a,b)|x)\eta(dz|x)\mu_\eta(dx)\\
	&=\int_{\td}\int_{\rdp}(\psi^+_{s,\eta,\psi}(x+2\vartheta_{s,\eta,\psi}a)+2\vartheta_{s,\eta,\psi}b) \omega_{s,\eta,\psi}(d(a,b)|x)\mu_\eta(dx)\\ &{}\hspace{100pt}+\int_{\td}\int_{\mathbb{R}}z\eta(dz|x)\mu_\eta(dx)\\&=
	[\psi^+_{s,\eta,\psi},\bar{\eta}^+_{s,\eta,\psi}]+\int_{\tdr}z\eta(d(x,z)).
	\end{split}
	\end{equation*} Analogously, 
	$$[\psi,\eta]=[\psi,\mu_\eta]+\int_{\tdr}z\eta(d(x,z)). $$
	
	Finally, let $s^+_{s,\eta,\psi}\triangleq s+2\vartheta_{s,\eta,\psi}$.
	
	One can summarize the properties of $(s^+_{s,\eta,\psi},\eta^+_{s,\eta,\psi},\psi^+_{s,\eta,\psi})$ and $\zeta_{s,\eta,\psi}$ as follows. 
	\begin{itemize}
		\item $(s^+_{s,\eta,\psi},\eta^+_{s,\eta,\psi},\psi^+_{s,\eta,\psi})\in \mathcal{Z}$;
		\item ${\eta}^+_{s,\eta,\psi}=\Xi^{(s^+_{s,\eta,\psi}-s)}{}_\#\zeta_{s,\eta,\psi}$;
		\item $\psi^+_{s,\eta,\psi}\in\mathcal{V}(s^+_{s,\eta,\psi},\mathrm{p}_\#{\eta}^+_{s,\eta,\psi})$;
		\item $|A^{s,s+\vartheta_{s,\eta,\psi}}_{\mathrm{p}_\#{\eta}}\psi^+_{s,\eta,\psi}-\psi|< (s^+_{s,\eta,\psi}-s)/(2n)$;
		\item $[\psi^+_{s,\eta,\psi},{\eta}^+_{s,\eta,\psi}]> [\psi,\eta]-(s^+_{s,\eta,\psi}-s)/(2n)$;
		\item $\int_{\tdr\times\rdp}\mathrm{dist}(v,\widehat{F}(t,w,\nu))\zeta_{s,\eta,\psi}(d(w,v))< 1/(2n).$
	\end{itemize}
	
	Let $\mathcal{E}_{s,\eta,\psi}$ be a set of triples $(t,\nu,\phi)\in \mathcal{Z}_n $ such that,   for some $\gamma\in\mathcal{L}_*^c(m)$, the following properties hold true:
	\begin{list}{(E\arabic{tmp})}{\usecounter{tmp}}
		\item\label{condE:time} $t\in (s-\vartheta_{s,\eta,\psi},s+\vartheta_{s,\eta,\psi})$;
		\item\label{condE:W} $W_1(\eta^+_{s,\eta,\psi},\Xi^{(s^+_{s,\eta,\psi}-t)}{}_\#\gamma)<(s^+_{s,\eta,\psi}-t)/n$;
		\item\label{condE:A} $\|A^{t,s^+_{s,\eta,\psi}}_{\mathrm{p}_\#\nu}\psi^{+}_{s,\eta,\psi}-\phi\|<(s^+_{s,\eta,\psi}-t)/n$;
		\item\label{condE:int} $[\psi^+_{s,\eta,\psi},\eta^+_{s,\eta_\psi}]> [\phi,\nu]-(s^+_{s,\eta,\psi}-t)/n$;
		\item\label{condE:F} $$\int_{\tdr\times\rdp}\mathrm{dist}(v,\widehat{F}(t,w,\nu))\gamma(d(w,v))<1/n. $$
	\end{list}
	
	Now, we shall prove that each set $\mathcal{E}_{s,\eta,\psi}$ is open in $\mathcal{Z}_n$.
	Let $(t_*,\nu_*,\phi_*)\in \mathcal{E}^n_{s,\eta,\psi}$. There exists  $\gamma_*\in \mathcal{L}_*^c(\nu_*)$ such that conditions (E1)--(E5) are fulfilled for $(t_*,\nu_*,\phi_*)$ and $\gamma_*$. Further, let $\delta>0$, and let $(t,\nu,\phi)\in \mathcal{Z}_n$ be such that $|t_*-t|<\delta$, $W_1(\nu_*,\nu)<\delta$, $\|\phi_*-\phi\|<\delta$. We shall show that, if $\delta$ is sufficiently small, then conditions (E1)--(E5) hold true for $(t,m,\phi)$ and
	$\gamma\triangleq \pi^0_{\nu,\nu_*}*\gamma_*$. Here $\pi^0_{\nu,\nu_*}$ stands for an optimal plan between $\nu$ and $\nu_*$, while $*$ denotes the  composition of probabilities  introduced by (\ref{intro:composition}). This will imply that $\mathcal{E}_{s,\eta,\psi}$ is open.
	
	First, notice that if $\gamma_*\in\mathcal{L}_*^c(\nu_*)$, then 	
	$\gamma=\pi^0_{\nu,\nu_*}*\gamma_*$ belongs to 	$\mathcal{L}_*^c(\nu)$. 
	
	Further,  condition (E\ref{condE:time}) holds for $(t,\nu,\phi)$ if $\delta<\vartheta_{s,\eta,\psi}-|t_*-s|$.  	
	
	Since $\gamma_*\in\mathcal{L}_*^c(\nu_*)$, by Lemma \ref{lm:Theta_pi} we have that $$|W_1(\Xi^{s^+_{s,\eta,\psi}-t}{}_\#\gamma,\Xi^{s^+_{s,\eta,\psi}-t_*}{}_\#\gamma_*)|\leq 2c\delta. $$ Therefore, condition (E\ref{condE:W}) is valid for $(t,\nu,\phi)$, if
	$$(2c+1/n)\delta<(s^+_{s,\eta,\psi}-t_*)/n-W_1(\eta^+_{s,\eta,\psi},\Xi^{s^+_{s,\eta,\psi}-t_*}{}_\#\gamma_*). $$
	
	Analogously, by Lemma \ref{lm:A_continuity} 
	\begin{equation*}\begin{split}
	\Bigl|\|A^{t,s^+_{s,\eta,\phi}}_{\mathrm{p}_\#\nu}\psi^+_{s,\eta,\psi}&-\phi\|- \|A^{t_*,s^+_{s,\eta,\phi}}_{\mathrm{p}_\#\nu_*}\psi^+_{s,\eta,\psi}-\phi_*\|\Bigr|\\&\leq \delta+ (C+1)[L3\vartheta_{s,\eta,\psi}\delta+R\delta+3\vartheta_{s,\eta,\psi}\alpha(\delta)]. \end{split}
	\end{equation*} Therefore, in the case when
	\begin{equation*}
	\begin{split}
	(1+L3\vartheta_{s,\eta,\psi}(C+1)&+R+1/n)\delta+3\vartheta_{s,\eta,\psi}(C+1)\alpha(\delta)\\
	&<(s^+_{s,\eta,\psi}-t_*)/n-\|A^{t_*,s^+_{s,\eta,\phi}}_{\mathrm{p}_\#\nu_*}\psi^+_{s,\eta,\psi}-\phi_*\| 
	\end{split}
	\end{equation*} condition (E\ref{condE:A}) holds for $(t,\nu,\phi)$. 
	
	By (\ref{estima:action_continuous}) we have that $|[\phi_*,\nu_*]-[\phi,\nu]|\leq (C+1)\delta$. Therefore, condition (E\ref{condE:int}) holds true for $(t,\nu,\phi)$ if 
	$$(C+1+1/n)\delta<[\psi^+_{s,\eta,\psi},\eta^+_{s,\eta,\psi}]- [\phi_*,\nu_*]+(s^+_{s,\eta,\psi}-t_*)/n. $$
	
	Finally, notice that $|\mathrm{dist}(v,\widehat{F}(t,w,\nu)-\mathrm{dist}(v,\widehat{F}(t,w,\nu_*)|\leq 2LW_1(\nu,\nu_*)$ and the function $(w,v)\mapsto \mathrm{dist}(v,\widehat{F}(t,w,\nu_*)$ is $(2L+1)$-Lipschitz continuous. Thus, by the second statement of Lemma \ref{lm:Theta_pi} we get
	\begin{equation*}
	\begin{split}
	\Bigl|\int_{\tdr\times\rdp}\mathrm{dist}(v,\widehat{F}(t,w,\nu_*))\gamma_*(d(w,v))-\int_{\tdr\times\rdp}\mathrm{dist}(v,\widehat{F}(t,w,\nu))&\gamma(d(w,v))\Bigr|
	\\ &\leq (4L+1)\delta. 
	\end{split}
	\end{equation*}
	Therefore, if one pick $\delta$ less than
	$$\frac{1}{4L+1}\left[\frac{1}{n}-\int_{\tdr\times\rdp}\mathrm{dist}(v,\widehat{F}(t,w,\nu_*))\gamma_*(d(w,v))\right], $$ then condition (E\ref{condE:F}) is valid for $(t,\nu,\phi)$ and $\gamma$.
	
	We have proved that if $(t_*,\nu_*,\phi_*)\in\mathcal{E}_{s,\eta,\psi}$, then its $\delta$-neighborhood in $\mathcal{Z}_n$ also lies in $\mathcal{E}_{s,\eta,\psi}$ for sufficiently small $\delta$. Furthermore, by construction $(s,\eta,\psi)\in \mathcal{E}_{s,\eta,\psi}$. Thus, $\{\mathcal{E}_{s,\eta,\psi}\}_{(s,\eta,\psi)\in\mathcal{Z}_n}$ is an open cover of $\mathcal{Z}_n$. Since $\mathcal{Z}_n$ is compact, there exists a finite number of triples $\{(s_i,\eta_i,\psi_i)\}_{i=1}^{I^n}$ such that
	$$\mathcal{Z}_n\subset \bigcup_{i=1}^{I^n}\mathcal{E}_{s_i,\eta_i,\psi_i}. $$ Put
	$$\theta_n\triangleq\min_{i=\overline{1,I^n}}\vartheta_{s_i,\eta_i,\psi_i}. $$ If $(t,\nu,\phi)\in \mathcal{Z}_n$, then there exists  $i=\overline{1,I^n}$ such that 
	$$(t,\nu,\phi)\in \mathcal{E}_{s_i,\eta_i,\psi_i}. $$ Choose $t^+\triangleq s^+_{s_i,\eta_i,\psi_i}$, $\nu^+\triangleq\eta^+_{s_i,\eta_i,\psi_i}$, $\phi^+\triangleq\psi^{+}_{s_i,\eta_i,\psi_i}$. Finally, let $\gamma$ be such that conditions (E1)--(E5) are fulfilled for $(t,\nu,\phi)$ and $\gamma$.
\end{proof}

For $\tau,\theta\in [0,T]$, $\tau<\theta$, define the mapping $\Lambda^{\tau,\theta}:\tdr\times\rdp\rightarrow \mathcal{C}_{\tau,\theta}$ by the rule: if $w\in \tdr$, $v\in\rdp$, then put
$\Lambda^{\tau,\theta}(w,v) $ be equal to the function $t\mapsto (w+(t-\tau)v)$.

Assume that $s,r\in [0,T]$, $s<r$, $m_*\in\ptd$, $\phi_*\in \mathcal{V}(s,m_*)$.  Put $r_n\triangleq r-1/n$. Without loss of generality we assume that $r_n>s$. Now, we construct a number $J_n$, a sequence of times $\{t^j_n\}_{j=0}^{J_n}$, sequences of probabilities $\{\nu_j^n\}_{j=0}^{J_n}$, $\{\eta^j_n\}_{j=0}^{J_n}$, $\{\gamma_n^j\}_{j=1}^{J_n}$, $\{\chi^j_n\}_{j=1}^{J_n}$ and a sequence of functions $\{\phi_j^n\}_{j=0}^{J_n}$ by the following rules.

\begin{itemize}
	\item Set $t^0_n\triangleq s$, $\nu^0_n=\eta^0_n\triangleq \widehat{m_*}$, $\phi^0_n\triangleq \phi_*$. 
	\item If $t^{j-1}_n$, $\nu^{j-1}_n$, $\eta^{j-1}_n$, $\phi^{j-1}_n$ are already constructed and $t^{j-1}_n\leq r_n$, then apply Lemma \ref{lm:step} with $t=t^{j-1}_n$, $\nu=\nu^{j-1}_n$, $\phi=\phi^{j-1}_n$ and choose $t^{j}_n$  equal to $t^+$, $\nu^{j}_n$  equal to $\nu^+$, $\phi^{j}_n$  equal to $\phi^+$, $\gamma^{j}_n$  equal to $\gamma$. Note that $\phi^{j}_n\in\mathcal{V}(t^{j-1}_n,\mathrm{p}_\#\nu^{j-1}_n)$, $\gamma^j_n\in\mathcal{L}^c_*(\nu_n^{j-1})$. Let $\pi_n^{j-1}$ be an optimal plan between $\eta^{j-1}_n$ and $\nu^{j-1}_n$. Set $\bar{\gamma}^{j}_n\triangleq \pi^{j-1}_n*\gamma^j_n$. Obviously, $\bar{\gamma}^j_n\in\mathcal{L}^c_*(\eta_n^{j-1})$. 
	Define the probability $\chi^{j}_n\in\pc{t^{j-1}_n}{t^{j}_n}$ by the rule: 
	\begin{equation}\label{intro:chi_j_n}
	\chi^{j}_n\triangleq\Lambda^{t^{j-1}_n,t^{j}_n}{}_\#\bar{\gamma}^{j}_n
	\end{equation} 
	Finally, set $\eta^{j}_n\triangleq \hat{e}_{t^{j}_n}{}_\#\chi^{j}_n=\Xi^{t^{j}_n-t^{j-1}_n}{}_\#\bar{\gamma}_n^j$.
	\item If $t^{j}_n>r_n$, then choose $J_n$ to be equal to $j$.
\end{itemize}
Notice that $t^{j+1}_n-t^{j}_n\geq\theta_n$. Thus, $J_n$ is finite.

Let ${\gamma}^{J_n+1}_n\in\mathcal{L}^c_*(\nu^{J_n}_n)$ be such that
$$\int_{\tdr\times\rdp} \mathrm{dist}(v,\widehat{F}(t^{J_n}_n,w,\nu^{J_n}_n))\gamma^{J_n+1}_n(d(x,v))=0. $$ Let $\bar{\gamma}^{J_n+1}_n\triangleq\pi^{J_n}_n*\gamma^{J_n+1}_n$, where $\pi^{J_n}_n$ is an optimal plan between $\eta^{J_n}_n$ and $\nu^{J_n}_n$.	Define the probability $\chi^{J_n+1}_n\in\pc{t^{J_n}_n}{r}$ by the rule: 
\begin{equation}\label{intro:chi_J_n_plus}
\chi^{J_n+1}_n\triangleq \Lambda^{t^{J_n}_n,r}{}_\#\bar{\eta}^{J_n+1}_n,\ \ j=\overline{0,J_n}.
\end{equation} 

We have that 
$$\phi_n^j\in\mathcal{V}(t_n^j,\mathrm{p}_\#\nu_n^j).  $$

Put
\begin{equation}\label{intro:chi_n}
\chi_n\triangleq \chi^1_n\odot\ldots\chi^{J_n}_n\odot\chi^{J_n+1}_n. 
\end{equation} Here $\odot$ stands for concatenation of probabilities (see (\ref{intro:concatination})). Since $\hat{e}_{t_n^j}{}_\#\chi_n^j=\hat{e}_{t_n^j}{}_\#\chi_n^{j+1}$ when $j=1,\ldots, J_n$, the probability $\chi_n$ is well-defined. Moreover, $\chi_n\in\pc{s}{r}$, $\hat{e}_{t^j_n}{}_\#\chi_n=\eta^j_n$. 

Let us point out the properties of the constructed sequences.
\begin{lemma}\label{lm:nu_j_mu_j} For $j=\overline{0,J_n}$,
	$$W_1(\eta^j_n,\nu^j_n)\leq (t^j_n-s)/n.$$
\end{lemma}
\begin{proof}
	First, we have that $W_1(\eta^0_n,\nu^0_n)=0$. Further, notice that $$\eta^j_n=\Xi^{t^j_n-t^{j-1}_n}{}_\#(\pi^{j-1}_n*\gamma^j_n).$$ Thus, if $W_1(\eta^{j-1}_n,\nu^{j-1}_n)\leq (t^{j-1}_n-s)/n$, then, using the second statement  of Lemma~\ref{lm:Theta_pi} and Lemma~\ref{lm:step}, we conclude that
	\begin{equation*}\begin{split}W_1(\eta^j_n,\nu^j_n)&= W_1(\Xi^{t^j_n-t^{j-1}_n}{}_\#(\pi^{j-1}_n*\gamma^j_n),\nu^j_n)\\ 
	&\leq 	W_1(\Xi^{t^j_n-t^{j-1}_n}{}_\#(\pi^{j-1}_n*\gamma^j_n),\Xi^{t^j_n-t^{j-1}_n}{}_\#\gamma^j_n)+ 
	W_1(\Xi^{t^j_n-t^{j-1}_n}{}_\#\gamma^j_n,\nu^j_n)\\&\leq
	(t^{j-1}_n-s)/n+(t^j_n-t^{j-1}_n)/n. \end{split}\end{equation*}
	
\end{proof}

Let $q_n(\cdot)\in\mathcal{M}_{s,r}$ be defined by the rule: $q_n(t)\triangleq e_t{}_\#\chi_n$.  Notice that \begin{equation}\label{equal:q_n_eta}
q_n(t_n^j)=\mathrm{p}_\#\eta_n^j,\ \ j=0,\ldots,J_n.
\end{equation}
\begin{lemma}\label{lm:B_distance} There exist constants $c_1$ and $c_2$ such that the following estimate holds true:
	$$\|B_{q_n(\cdot)}^{s,r}\phi_n^{J_n}-\phi_*\|\leq c_1\alpha(1/n)+c_2/n.$$
\end{lemma}
\begin{proof} To simplify designations put $m_n^j\triangleq \mathrm{p}_\#\nu_n^j$.
	
	Let us show that
	\begin{equation}\label{A_j_estima:pr_lm_B_dist}
	\|A^{t^j_n,t^{j+1}_n}_{m^j_n}\ldots A^{t^{J_n-1}_n,t^{J_n}_n}_{m^{J_n-1}_n}\phi^{J_n}_n-\phi^j_n\|\leq (t^{J_n}_n-t^{j}_n)/n.
	\end{equation} We prove this inequality by the backward induction on $j$.
	By Lemma \ref{lm:step} we have that $$\|A^{t^{J_n-1}_n,t^{J_n}_n}_{m^{J_n-1}_n}\phi^{J_n}_n-\phi_n^{J_n-1}\|\leq
	(t^{J_n}_n-t^{J_n-1}_n)/n. $$
	Now, assume that that inequality (\ref{A_j_estima:pr_lm_B_dist}) holds true for some $j$. We have that
	\begin{equation*}
	\begin{split}
	\|A^{t^{j-1}_n,t^{j}_n}_{m^{j-1}_n}&A^{t^j_n,t^{j+1}_n}_{m^j_n}\ldots A^{t_{J_n-1}^n,t^{J_n}_n}_{m^{J_n-1}_n}\phi^{J_n}_n-\phi^{j-1}_n\|\\&\leq
	\|A^{t^{j-1}_n,t^{j}_n}_{m^{j-1}_n}A^{t^j_n,t^{j+1}_n}_{m^j_n}\ldots  A^{t^{J_n-1}_n,t^{J_n}_n}_{m^{J_n-1}_n}\phi^{J_n}_n- A^{t^{j-1}_n,t^{j}_n}_{m^{j-1}_n}\phi^{j}_n\|+\|A^{t^{j-1}_n,t^{j}_n}_{m^{j-1}_n}\phi^{j}_n- \phi^{j-1}_n\|
	\end{split}
	\end{equation*}
	By Lemma \ref{lm:step} we have that 
	$$\|A_{m_n^{j-1}}^{t_n^{j-1},t_n^j}\phi_n^j-\phi_n^{j-1}\|\leq (t_n^j-t_n^{j-1})/n. $$
	This, Lemma \ref{lm:A_continuity} and assumption (\ref{A_j_estima:pr_lm_B_dist}) imply the inequality:
	$$
	\|A^{t^{j-1}_n,t^{j}_n}_{m^{j-1}_n}A^{t^j_n,t^{j+1}_n}_{m^j_n}\ldots A^{t^{J_n-1}_n,t^{J_n}_n}_{m^{J_n-1}_n}\phi^{J_n}_n-\phi^{j-1}_n\|\\\leq
	(t^{J_n}_n-t^{j}_n)/n+(t^{j}_n-t^{j-1}_n)/n. $$ Hence (\ref{A_j_estima:pr_lm_B_dist}) is fulfilled for any $j=0,\ldots, J_n-1$.
	
	Definition of $q_n(\cdot)$, (\ref{equal:q_n_eta}) and Lemma \ref{lm:nu_j_mu_j} give, for each $j$, the estimate
	\begin{equation}\label{estima:q_n_m_n_distance}
	\sup_{t\in [t^j_n,t^{j+1}_n]}W_1(q_n(t),m^j_n)\leq (T+c)/n.
	\end{equation}
	Further, put
	$$C^*(s)\triangleq[(C +1)e^{L(T-s)}-1].$$
	Now, let us prove that, for any $j=0,\ldots, J_n-1$,
	\begin{equation}\label{ineq:A_B_j}
	\begin{split}
	\|B^{t^j_n,t^{J_n}_n}_{q_n(\cdot)}&\phi^{J_n}_n-A^{t^j_n,t^{j+1}_n}_{m^j_n}\ldots A^{t^{J_n-1}_n,t^{J_n}_n}_{m^{J_n-1}_n}\phi^{J_n}_n\|\\&\leq
	(C^*(t^{j+1}_n)+1)(\alpha(1/n)+L(R+T+c)/n+T/n)(t^{J_n}_n-t^j_n).
	\end{split}
	\end{equation} As above, we prove this inequality by the backward induction.
	First, since $C(t_n^{J_n})\geq C$, by Lemma \ref{lm:bellman_close} and estimate (\ref{estima:q_n_m_n_distance}) we have that 
	\begin{equation*}\begin{split}
	\|B^{t^{J_n-1}_n,t_n^{J_n}}_{q_n(\cdot)}&\phi_n^{J_n}-
	A^{t^{J_n-1}_n,t^{J_n}_n}_{m^{J_n-1}_n}\phi^{J_n}_n\|
	\\&\leq (C^*(t^{J_n}_n)+1)(\alpha(1/n)+L(R+T+c)/n)(t^{J_n}_n-t^{J_n-1}_n). \end{split}
	\end{equation*}
	
	Further, assume that (\ref{ineq:A_B_j}) is fulfilled for some $j$. We have that
	$$B^{t^{j-1}_n,t^{J_n}_n}_{q_n(\cdot)}\phi^{J_n}_n= B^{t^{j-1}_n,t^{j}_n}_{q_n(\cdot)}B^{t^{j}_n,t^{J_n}_n}_{q_n(\cdot)}\phi^{J_n}_n. $$ Additionally, by Lemmas \ref{lm:lipschitz} and \ref{lm:A_lipshitz} we have that
	$$B^{t^{j}_n,t^{J_n}_n}_{q_n(\cdot)}\phi^{J_n}_n,A^{t^j_n,t^{j+1}_n}_{m^j_n}\ldots A^{t^{J_n-1}_n,t^{J_n}_n}_{m^{J_n-1}_n}\phi^{J_n}_n\in\mathrm{Lip}_{C^*(t^j_n)}(\td). $$
	This, assumption of the induction (see (\ref{ineq:A_B_j})) and Lemma \ref{lm:bellman_close} imply that
	\begin{equation*}
	\begin{split}
	\Bigl\|B^{t^{j-1}_n,t^{J_n}_n}_{q_n(\cdot)}\phi^{J_n}_n-&
	A^{t^{j-1}_n,t^{j}_n}_{m^j_n}A^{t^j_n,t^{j+1}_n}_{m^j_n}\ldots A^{t^{J_n-1}_n,t^{J_n}_n}_{m^{J_n-1}_n}\phi^{J_n}_n\Bigr\|\\ = 
	\Bigl\|B^{t^{j-1}_n,t^{j}_n}_{q_n(\cdot)}(&B^{t^{j}_n,t^{J_n}_n}_{q_n(\cdot)}\phi^{J_n}_n)-
	A^{t^{j-1}_n,t^{j}_n}_{m^j_n}(A^{t^j_n,t^{j+1}_n}_{m^j_n}\ldots A^{t^{J_n-1}_n,t^{J_n}_n}_{m^{J_n-1}_n}\phi^{J_n}_n)\Bigr\|
	\\
	\leq (C^*(t^{j}_n)+&1)(\alpha(1/n)+L(R+T+c)/n+T/n)(t^{j}_n-t^{j-1}_n)\\&+ (C^*(t^{j+1}_n)+1)(\alpha(1/n)+L(R+T+c)/n+T/n)(t^{J_n}_n-t^j_n).
	\end{split}
	\end{equation*}
	Since the function $s\mapsto C^*(s)$ decrease, we conclude that (\ref{ineq:A_B_j}) is fulfilled for all $j$. 
	
	Since $t^0_n=s$, combining (\ref{A_j_estima:pr_lm_B_dist}) and (\ref{ineq:A_B_j}), we get
	\begin{equation}\label{ineq:B_psi_J_n_phi_star}
	\|B^{s,t^{J_n}_n}_{q_n(\cdot)}\phi^{J_n}_n-\phi_*\|\leq (C^*(0)+1)(\alpha(1/n)+(LR+LT+Lc+T)/n).
	\end{equation} 
	Finally, using Lemma \ref{lm:B_continuity},  we have that 
	$$\|B^{s,r}_{q_n(\cdot)}\phi^{J_n}_n-B^{s,t^{J_n}_n}_{q_n(\cdot)}\phi^{J_n}_n\|\leq (C+1)R(r-t^{J_n}_n)\leq (C+1)R/n. $$
	This and inequality (\ref{ineq:B_psi_J_n_phi_star}) yield the conclusion of the lemma.
\end{proof}

\begin{lemma}\label{lm:action_n}
	The following estimate holds true, for some constant $c_3>0$:
	$$[\phi_n^{J_n},\eta_n^{J_n}]>[\phi_*,\widehat{m_*}]-c_3/n. $$
\end{lemma}
\begin{proof}
	First, notice that by construction
	$$[\phi^j_n,\nu_n^j]>[\phi^{j-1}_n,\nu_n^{j-1}]-(t_n^{j}-t_n^{j-1})/n. $$ Therefore, since $\phi^0_n=\phi_*$, $\nu^0_n=\widehat{m_*}$, $s=t_n^0$, $r>t^{J_n}_n$, we get
	$$[\phi^{J_n}_n,\nu_n^{J_n}]>[\phi_*,\widehat{m_*}]-(r-s)/n. $$
	Recall that $\phi^{J_n}_n\in\mathrm{Lip}_C(\td)$. Hence, using (\ref{estima:action_continuous}) and Lemma \ref{lm:nu_j_mu_j}, we conclude that
	$$[\phi_n^{J_n},\eta_n^{J_n}]>[\phi_*,\widehat{m_*}]-(C+1)(r-s)/n. $$
\end{proof}

\begin{proof}[Proof of Theorem \ref{th:viability}. Sufficiency]
	
	We will prove that there exists $(\mu,\psi)$ such that $\psi\in \mathcal{V}(r,\mu)$ and $(\mu,\psi)\in \Psi^{r,s}(m_*,\phi_*)$.
	By construction we have that, for each $n$, $\mathrm{supp}(\chi_n)$ lies in the compact set consisting of absolutely continuous trajectories $(x(\cdot),z(\cdot))\in \mathcal{C}_{s,r}$ such that $\|\dot{x}(t)\|,|\dot{z}(t)|\leq c\text{ for a.e. }t\in [s,r]$. 
	Additionally, elements of $\mathrm{supp}(\chi_n)$ are uniformly bounded. Thus, by~\cite[Proposition 7.1.5]{Ambrosio} we conclude that $\{\chi_n\}$ is relatively compact. This means that there exists a sequence $\{n_k\}$ and a probability $\chi\in\mathcal{C}_{0,T}$ such that
	\begin{equation}\label{conv:chi_n_k_chi}
	W_1(\chi_{n_k},\chi)\rightarrow 0\text{ as }k\rightarrow \infty.
	\end{equation} 
	
	Further, put
	\begin{equation*}\label{intro_fin_suff:m}
	\nu(t)\triangleq \hat{e}_t{}_\#\chi,\ \ m(t)\triangleq \mathrm{p}_\#\nu(t),\ \  \mu\triangleq m(r). 
	\end{equation*}
	
	Recall that $\{\phi_n^{J_n}\}\subset\mathrm{Lip}_C(\td)$ and  $\|\phi^{J_n}_n\|\leq M$. Thus, $\{\phi_{n}^{J_n}\}$ is relatively compact.  We can assume without loss of generality that the sequence $\{\phi_{n_k}^{J_{n_k}}\}$ converges to a function $\psi\in \mathrm{Lip}_C(\td)$. By (\ref{ineq:Wasserstein_e_t}) we have that
	\begin{equation}\label{ineq:m_q_nk}
	W_1(m(t),q_{n_k}(t))\leq W_1(\chi_{n_k},\chi). 
	\end{equation} Additionally, $r-t_{n_k}^{J_{n_k}}\leq 1/n$,
	$W_1(q_{n_k}(t'),q_{n_k}(t''))\leq c|t'-t''|$, $W_1(q_{n_k}(t_{n_k}^{J_{n_k}}),\mathrm{p}_\# \nu_{n_k}^{J_{n_k}} )\leq W_1(\eta_{n_k}^{J_{n_k}},\nu_{n_k}^{J_{n_k}})\leq (r-s)/n$. Since
	$\phi_{n_k}^{J_{n_k}}\in\mathcal{V}(t_{n_k}^{J_{n_k}},\mathrm{p}_\# \nu_{n_k}^{J_{n_k}})$ and $\mathcal{V}$ is upper semicontinuous, we have that
	\begin{equation}\label{incl_fin_suff:psi_limit_V}
	\psi\in \mathcal{V}(r,\mu).
	\end{equation}
	
	To show that $(\mu,\psi)\in \Psi^{r,s}(m_*,\phi_*)$ let us check that $\nu(\cdot)$ solves (\ref{system:mfdi}) and $\nu(\cdot)$, $\psi$ satisfy conditions ($\Psi$\ref{psi_cond:init_bound})--($\Psi$\ref{psi_cond:int_ineq}).
	
	To prove that $\nu(\cdot)$ is a solution of (\ref{system:mfdi}) pick $\tau_0,\tau_1\in [s,r]$. There exist numbers $I_n^0$, and $I_n^1$ such that $\tau_0\in [t_n^{I^0_n-1},t_n^{I_n^0}]$, $\tau_1\in [t_n^{I^1_n-1},t_n^{I_n^1}]$. Without loss of generality we can assume that $I_n^0<I_n^1$. Put $h_n^{I_0-1}\triangleq \tau_0$. For $k=I_n^0,\ldots,I_n^1-1$, put $h_n^k\triangleq t_n^k$. Set $h_n^{I_0^1}\triangleq \tau_1$. Further, denote $\delta_n^k\triangleq h_n^k-h_{n}^{k-1}$. Additionally, let $\varpi_n(t)\triangleq \hat{e}_t{}_\#\chi_n$. Notice that $q_n(t)=\mathrm{p}_\#\varpi_n(t)$. For any $w(\cdot)\in\mathrm{supp}(\chi_n)$, we have that
	\begin{equation*}
	\begin{split}
	\mathrm{dist}&\Bigl(w(\tau_1)-w(\tau_0),\int_{\tau_0}^{\tau_1}\widehat{F}(\theta,w(\theta),\varpi_n(\theta))d\theta \Bigr)\\&\leq
	\sum_{k=I_n^0}^{I_1^n}\mathrm{dist}\Bigl(w(h_n^k)-w(h_n^{k-1}), \int_{h_n^{k-1}}^{h_n^{k}}\widehat{F}(\theta,w(\theta),\varpi_n(\theta))d\theta \Bigr)\\ 
	&\leq
	\sum_{k=I_n^0}^{I_1^n}\mathrm{dist}\Bigl(w(h_n^k)-w(h_n^{k-1}), \delta_n^{k}\widehat{F}(t_n^{k-1},w(t_n^{k-1}),\varpi_n(t_n^{k-1}))\Bigr)
	\\&{}\hspace{80pt}+(\tau_n^1-\tau_n^0)[\alpha(1/n)+2Lc/n].
	\end{split}
	\end{equation*} 
	
	Therefore, by the construction of $\chi_n$ (see \ref{intro:chi_n}) and $\chi_n^k$ (see (\ref{intro:chi_j_n}) and (\ref{intro:chi_J_n_plus}))
	\begin{equation*}
	\begin{split}
	\intc{\tau_0}{\tau_1}\mathrm{dist}\Bigl(w(\tau_1)-w(\tau_0), &\int_{\tau_0}^{\tau_1}\widehat{F}(\theta,w(\theta),\varpi_n(\theta))d\theta \Bigr)\chi_n(d(w(\cdot)))\\ \leq
	\sum_{k=I_n^0}^{I_1^n}\intc{\tau_0}{\tau_1}\mathrm{dist}\Bigl(w(h_n^k)&-w(h_n^{k-1}), \delta_n^{k}\widehat{F}(t_n^{k-1},w(t_n^{k-1}),\varpi_n(t_n^{k-1}))\Bigr)\chi_n^k(d(w(\cdot)))\\ 
	&+(\tau_n^1-\tau_n^0)[\alpha(1/n)+2Lc/n]\\ = 
	\sum_{k=I_n^0}^{I_1^n}\int_{\tdr\times\rdp} \delta_n^{k}\mathrm{dist}&\Bigl(v, \widehat{F}(t_n^{k-1},w,\eta_n^{k-1}\Bigr)\bar{\gamma}_n^k(d(w,v))
	\\&+(\tau_n^1-\tau_n^0)[\alpha(1/n)+2Lc/n].
	\end{split}
	\end{equation*} 
	
	Since functions $(w,v)\mapsto \mathrm{dist}(v,\widehat{F}(t,w,\nu))$, $\nu\mapsto \mathrm{dist}(v,\widehat{F}(t,w,\nu))$ are Lipschitz continuous with constants $(2L+1)$ and $2L$ respectively, by the second statement of Lemma \ref{lm:Theta_pi},  Lemma \ref{lm:nu_j_mu_j} and choice of $\bar{\gamma}_n^k$, ${\gamma}_n^k$ we have that
	\begin{equation*}
	\begin{split}
	\intc{\tau_0}{\tau_1}\mathrm{dist}\Bigl(w(\tau_1)-w(\tau_0), &\int_{\tau_0}^{\tau_1}\widehat{F}(\theta,w(\theta),\varpi_n(\theta))d\theta \Bigr)\chi_n(d(w(\cdot)))\\ \leq
	\sum_{k=I_n^0}^{I_1^n}\int_{\tdr\times\rdp} \delta_n^{k}\mathrm{dist}&\Bigl(v, \widehat{F}(t_n^{k-1},w,\nu_n^{k-1}\Bigr){\gamma}_n^k(d(w,v))
	\\&+(\tau_n^1-\tau_n^0)[\alpha(1/n)+4Lc/n+(1+4L)T/n]\\ \leq (\tau_n^1-\tau_n^0)[\alpha(1/n)+4&Lc/n+(1+4L)T/n+1/n].
	\end{split}
	\end{equation*}
	
	Hence, using the fact that the function $w(\cdot)\mapsto \mathrm{dist}(w(\tau_1)-w(\tau_0), \int_{\tau_0}^{\tau_1}\widehat{F}(\theta,w(\theta),\varpi_n(\theta))d\theta)$ is $(2+2L(\tau_1-\tau_0))$-Lipschitz continuous and inequality (\ref{ineq:m_q_nk}), we get that
	\begin{equation*}
	\begin{split}
	\intc{\tau_0}{\tau_1}\mathrm{dist}\Bigl(w(\tau_1)-w(\tau_0), &\int_{\tau_0}^{\tau_1}\widehat{F}(\theta,w(\theta),\nu(\theta))d\theta \Bigr)\chi(d(w(\cdot)))\\ 
	\leq
	\intc{\tau_0}{\tau_1}\mathrm{dist}\Bigl(w(\tau_1)&-w(\tau_0), \int_{\tau_0}^{\tau_1}\widehat{F}(\theta,w(\theta),\varpi_{n_k}(\theta))d\theta \Bigr)\chi_{n_k}(d(w(\cdot)))\\+(2+4L(\tau_1&-\tau_0))W_1(\chi,\chi_{n_k})\\ \leq (2+2L(\tau_1-\tau_0))&W_1(\chi,\chi_{n_k})\\
	+(\tau_{n_k}^1-\tau_{n_k}^0)[&\alpha(1/n_k)+4Lc/n_k+(1+4L)T/n_k+1/n_k].
	\end{split}
	\end{equation*}
	Passing to the limit when $k\rightarrow\infty$, using the equality $F(\theta,\mathrm{p}(w(\theta)),\mathrm{p}_\#\nu(\theta))=\widehat{F}(\theta,w(\theta),\nu(\theta))$ and convergence (\ref{conv:chi_n_k_chi}), we conclude that, for any $\tau_0,\tau_1\in [s,r]$,
	$$\intc{\tau_0}{\tau_1}\mathrm{dist}\left(w(\tau_1)-w(\tau_0), \int_{\tau_0}^{\tau_1}F(\theta,\mathrm{p}(w(\theta)),m(\theta)d\theta \right)\chi(d(w(\cdot)))=0. $$ This means that, for $\chi$-a.e. $w(\cdot)\in\mathcal{C}_{s,r}$,
	$w(\cdot)\in \mathrm{SOL}(r,s,m(\cdot))$. This and definition of $\nu(\cdot)$ imply that
	$\nu(\cdot)$ solves (\ref{system:mfdi}) on $[s,r]$.
	
	Further, we have that property ($\Psi$\ref{psi_cond:init_bound}) is fulfilled by construction. Property ($\Psi$\ref{psi_cond:B}) follows from Lemmas \ref{lm:B_continuity}, \ref{lm:B_distance} and convergences $$\|\psi-\phi_{n_k}^{J_{n_k}}\|, \sup_{t\in [s,r]}W_1(q_{n_k}(t),m(t))\rightarrow 0 \text{ as }k\rightarrow \infty.$$
	
	Finally, convergence of $\eta_{n_k}^{J_{n_k}}$ to $\nu(r)$ and Lemma \ref{lm:action_n} yield ($\Psi$\ref{psi_cond:int_ineq}).
	
	Thus, we prove that $(\mu,\psi)\in \Psi^{r,s}(m_*,\phi_*)$. This and inclusion
	$\psi\in\mathcal{V}(r,\mu)$ (see (\ref{incl_fin_suff:psi_limit_V})) implies the sufficiency part of Theorem \ref{th:viability}.
\end{proof}

\section{Proof of the viability theorem. Necessity}\label{sect:proof_necissity}
\begin{proof}[Proof of Theorem \ref{th:viability}. Necessity]
	Let $\mathcal{V}$ be viable with respect to the mean field game dynamics. Choose $s\in [0,T]$, $m\in \td$, $\phi\in\mathcal{V}(s,m)$. By Definition \ref{def:dpp}, for any $r>s$, there exist  flows of probabilities $\nu^r(\cdot)\in\mathcal{N}_{s,r}$, $m^r(\cdot)\in\mathcal{M}_{s,r}$,  and a function $\psi^r\in C(\td)$ such that $\nu^r(\cdot)$ solves (\ref{system:mfdi}), $m^r(t)=\mathrm{p}_\#\nu^r(t)$ and conditions ($\Psi$\ref{psi_cond:init_bound})--($\Psi$\ref{psi_cond:int_ineq}) hold true. Furthermore, $\psi^r\in \mathcal{V}(r,m^r(r))$. Without loss of generality one can assume that $\nu^r(s)=\widehat{m}$. 
	
	Let $\chi^r\in\mathcal{C}_{s,r}$ be such that $\nu^r(t)=\hat{e}_t{}_\#\chi^r$ and $\mathrm{supp}(\chi^r)\subset\mathrm{SOL}(r,s,m^r(\cdot))$. 
	
	Further, we define the operator $\Delta^{s,r}:\mathcal{C}_{s,r}\rightarrow \td\times\rdp$ by the rule:
	\begin{equation}\label{intro:Delta}
	\Delta^{s,r}(x(\cdot),z(\cdot))=\left(x(s),\frac{x(r)-x(s)}{r-s},\frac{z(r)-z(s)}{r-s}\right). 
	\end{equation}
	
	Set
	$$\beta^r\triangleq \Delta^{s,r}{}_\#\chi^r. $$ 
	Notice that $\beta^r\in\mathcal{L}^c(m)$ for $c=R$. 
	Since $\|B_{m^r(\cdot)}^{s,r}\psi^r-\phi\|=0,$ using Lemma \ref{lm:bellman_close} we have that
	\begin{equation}\label{A_m_r}
	\|A_{m}^{s,r}\psi^r-\phi\|\leq (C+1)(\alpha(r-s)+2LR(r-s))(r-s).
	\end{equation} Notice that  $\Theta^{r-s}{}_\#\beta^r=\nu^r(r)$. Furthermore, by construction we have that
	\begin{equation}\label{ineq:beta_r_psi}
	[\phi^r,\nu^r(r)]\geq [\phi,\widehat{m}].
	\end{equation}
	
	For any $w(\cdot)=(x(\cdot),z(\cdot))\in \mathrm{supp}(\chi^r)$, we have that
	\begin{equation*}
	\begin{split}
	0&=\mathrm{dist}\left(w(r)-w(s),\int_s^rF(t,x(t),m^r(t))dt\right)\\&\geq \mathrm{dist}\left(w(r)-w(s),(r-s)F(s,x(s),m(s))\right)-(r-s)[\alpha(r-s)+4LR(r-s)].
	\end{split}
	\end{equation*} Thus,
	$$\mathrm{dist}\left(\frac{w(r)-w(s)}{r-s},F(s,x(s),m(s))\right)\leq \alpha(r-s)+4LR(r-s). $$
	Hence we have
	\begin{equation}\label{ineq:F_beta_r}
	\intTdRdp \mathrm{dist}\left(v,F(s,x,m)\right)\beta^r(d(x,v))\leq \alpha(r-s)+4LR(r-s).
	\end{equation}
	
	Since the set $\mathcal{L}^c(m)$ is compact in $\mathcal{P}^1(\td\times\rdp)$ we have that there exists a probability $\beta$ and a sequence $\{r_n\}_{n=1}^\infty$ such that
	$r_n\rightarrow s$, $W_1(\beta^{r_n},\beta)\rightarrow 0$ as $n\rightarrow\infty$.
	
	Let us show that the probability $\beta\in \mathcal{D}^c_F\mathcal{V}(s,m,\phi)$. Letting $\tau_n\triangleq r_n-s$ we get the first condition. We have that
	$\psi^{r_n}\in\mathcal{V}(r^n,\Theta^{r_n-s}{}_\#\beta^{r_n})$. Thus, the second condition of Definition \ref{def:tangent} is fulfilled. The third condition of Definition \ref{def:tangent} follows from (\ref{A_m_r}). Analogously, (\ref{ineq:beta_r_psi}) implies Condition 4 for $\nu^n\triangleq \nu^{r_n}(r_n)$. Finally, passing to the limit in (\ref{ineq:F_beta_r}) we get the last condition of Definition \ref{def:tangent}.
\end{proof}

\section{Concluding remarks}
In the paper we examined the value multifunction of the mean field game. It assigns to each time and probability on the phase space the set of continuous functions such that every function from this set is a reward of the representative player corresponding to a solution of the mean field game.  We find the  sufficient condition on a given upper semicontinuous set-valued mapping to be a value multifunctions. This condition is expressed in the terms of the viability theory. Moreover, we obtain the infinitesimal form of this condition. It involves the notion of the set-valued derivative  with respect to the mean field game dynamics. The results of the paper can be regarded  as an extension of the master equation of the mean field game in the framework of the viability theory. The proposed approach works even in the case when  the solution of the mean field game is nonunique or discontinuously depends on the initial distribution of players.

In the paper we considered only deterministic mean field games. The extension of the results of the paper to the  stochastic mean field games and especially to the mean field games with common noise is  a subject of a future research.

Notice that in the general case the value multifunction is nonunique. Thus, it is natural to consider the maximal  value multifunction. The maximal value multifunction is always viable under the mean field game dynamics. 
The search of the  characterization of the maximal value function is also a theme of future research. 

In the paper we use a set-valued derivative of the multifunction depending on time and probability. When the value multifunction is reduced to the usual function it is natural question to find the relation between the set-valued derivative on the one hand  and existing approaches to the smooth and nonsmooth analysis in the space of probability (see, for details, \cite{Ambrosio}, \cite{Cardaliaguet_Delarue_Lasry_Lions_2015}, \cite{Carmona_Delarue_I}, \cite{Carmona_Delarue_II}, \cite{Gangbo_Nguen_Tudorascu}, \cite{Gangbo_Tudorascu}, \cite{Marigonda_Quincampoix}) on the other hand. The answer to this question will clarify the relation  between the master equation and the results of the paper.

Finally, it is worth to mention that the master equation was used in \cite{Cardaliaguet_Delarue_Lasry_Lions_2015} to derive the convergence of feedback Nash equilibria for $N$ player games to the solution of the limiting mean field game under Lasry-Lions monotonicity condition which assures the uniqueness of solution of the mean field game. It is intriguing to apply the results of this paper to obtain analogs of the mentioned convergence  for the deterministic mean field  games with the multiplicity of solutions.

\section*{Appendix A. Properties of the solutions to mean field type differential inclusions}
\setcounter{theorem}{0}
\setcounter{equation}{0}
\renewcommand{\theequation}{A\arabic{equation}}
\begin{proof}[Proof of Proposition \ref{prop:shift}]
	Let $\chi\in\mathcal{P}^1(\mathcal{C}_{s,r})$ be such that $\nu(t)=\hat{e}_t{}_\#\chi$, and let, for $m(t)\triangleq\mathrm{p}_\#\nu(t)=e_t{}_\#\chi$, $\mathrm{supp}(\chi)\subset \mathrm{SOL}(r,s,m(\cdot))$. Now, we define the probability $\bar{\chi}\in\mathcal{P}^1(\mathcal{C}_{s,r})$ in the following way. Let 	 $\{\chi_{x}\}$ be a disintegration of $\chi$ along $m(s)$ i.e. for every $\varphi\in C_b(\mathcal{C}_{s,r})$,
	$$\int_{\mathcal{C}_{s,r}}\varphi(w(\cdot))\chi(d(w(\cdot)))=\int_{\td}\int_{\mathcal{C}_{s,r}}\varphi(w(\cdot))\chi_{x}(d(w(\cdot)))m(s,dx). $$ For $\varphi\in C_b(\mathcal{C}_{s,r})$, set
	\begin{equation*}
	\begin{split}
	\int_{\mathcal{C}_{s,r}}&\varphi(w(\cdot))\bar{\chi}(d(w(\cdot)))\\&\triangleq  \int_{\td}\int_{\mathcal{C}_{s,r}} \int_{\mathbb{R}}\varphi(x(\cdot),z(\cdot)-z(s)+z_*)\chi_{x}(d(x(\cdot),z(\cdot))) \nu_*(dz_*|x)m(s,dx). \end{split}
	\end{equation*}
	
	Put $\bar{\nu}\triangleq \hat{e}_t{}_\#\chi$. Obviously, $\mathrm{p}_\#\bar{\nu}(t)=m(t)=\mathrm{p}_\#\nu(t)$, $\bar{\nu}(s)=\nu_*$. Thus, statements~2 and 3 of the proposition are obtained.
	
	Now, we shall prove that $\bar{\nu}(\cdot)$ solves MFDI (\ref{system:mfdi}). To this end let us show the inclusion $\mathrm{supp}(\bar{\chi})\subset\mathrm{SOL}(r,s,m(\cdot))$. We have that $\bar{\chi}$ is concentrated on the set of motions $(x(\cdot),z(\cdot)+h)$, where $(x(\cdot),z(\cdot))$ solves (\ref{incl:F}) and $h$ does not depend on $t$. Therefore the motion $(x(\cdot),z(\cdot)+h)$ itself is a solution of (\ref{incl:F}). Thus, $\mathrm{supp}(\bar{\chi})\subset\mathrm{SOL}(r,s,m(\cdot))$. This proves the first statement of the proposition.
	
	Finally, for $\phi\in C(\td)$ and $t\in [s,r]$, we have
	\begin{equation*}
	\begin{split}
	[\phi,&\bar{\nu}(t)]=\int_{\tdr}(\phi(x)+z)\bar{\nu}(t,d(x,z))\\
	&=\int_{\mathcal{C}_{s,r}}(\varphi(x(t))+z(t))\bar{\chi}(d(x(\cdot),z(\cdot)))\\&= \int_{\td} \int_{\mathcal{C}_{s,r}}\int_{\mathbb{R}}(\phi(x(t))+z(t)-z(s)+z_*)\nu_*(dz|x)\chi_{x}(d(x(\cdot),z(\cdot)))m(s,dx)\\
	&=\int_{\mathcal{C}_{s,r}}\int_{\mathbb{R}}(\phi(x(t))+z(t))\chi(d(x(\cdot),z(\cdot)))\\
	&{}\hspace{40pt}-\int_{\mathcal{C}_{s,r}}\int_{\mathbb{R}}z(s)\chi(d(x(\cdot),z(\cdot)))+\int_{\tdr}z\nu_*(d(x,z))\\
	&=\int_{\tdr}(\phi(x)+z)\nu(t,d(x,z))-\int_{\tdr}z\nu(s,d(x,z))+\int_{\tdr}z\nu_*(d(x,z)).
	\end{split}
	\end{equation*} This gives the forth statement of the proposition. 
	
\end{proof}

\begin{proof}[Proof of Proposition \ref{prop:concatination}] For $i=1,2$, pick $\chi_i\in\mathcal{P}^1(\mathcal{C}_{s_{i-1},s_i})$ such that $\nu_i(t)=\hat{e}_t{}_\#\chi_i$, and, for $m_i(t)\triangleq \mathrm{p}_\#\nu_i(t)$, $\mathrm{supp}(\chi_i)\in\mathrm{SOL}(s_i,s_{i-1},m(\cdot))$. Put $\chi\triangleq \chi_1\odot\chi_2$, $\nu(t)\triangleq\hat{e}_t{}_\#\chi$. Further, denote $m(t)\triangleq \mathrm{p}_\#\nu(t)$. Since 	$$m(t)\triangleq \left\{\begin{array}{cc}
	m_1(t), & t\in [s_0,s_1];\\
	m_2(t), & t\in [s_1,s_2],
	\end{array}\right. $$ we have that $\mathrm{supp}(\chi)\subset\mathrm{SOL}(s_0,s_2,m(\cdot))$. This proves the proposition.
\end{proof}

\begin{proof}[Proof of Proposition \ref{prop:limit}]
	Pick $\chi_i\in\mathcal{P}^1(\mathcal{C}_{s,r})$ such that $\nu_i(t)=\hat{e}_t{}_\#\chi_i$ and $\mathrm{supp}(\chi)\subset \mathrm{SOL}(r,s,m_i(\cdot))$ where $m_i(t)=\mathrm{p}_\#\nu_i(t)$. Notice  that, for every $i$, $\chi_i$ is concentrated on the set of $2R$-Lipschitz continuous motions. Moreover, if $w(\cdot)\in \mathrm{supp}(\chi_i)$, then
	\begin{equation*}\label{A:ineq:norm}
	\|w(\tau_i)\|-2R(r-s)\leq \|w(\cdot)\|\leq \|w(\tau_i)\|+2R(r-s).
	\end{equation*} Hence, taking into account the equality $\hat{e}_{\tau_i}{}_\#\chi_i=\nu_i^\natural\in\mathcal{P}^1(\tdr)$, and the convergences $\tau_i\rightarrow\tau$, $W_1(\nu_i^\natural,\nu^\natural)\rightarrow 0$, we get that the sequence $\{\chi_i\}$ is tight and has uniformly integrable 1-moment. Thus, there exists a subsequence $\{\chi_{i_k}\}$ and a probability $\chi^*\in\mathcal{P}^1(\mathcal{C}_{s,r})$ such that
	$$\lim_{k\rightarrow\infty} W_1(\chi_{i_k},\chi^*)=0. $$ Put
	$\nu^*(t)\triangleq \hat{e}_t{}_\#\chi^*$, $m^*(t)\triangleq \mathrm{p}_\#\nu^*(t)=e_t{}_\#\chi^*$. By (\ref{ineq:Wasserstein_e_t}) we have that
	$$\lim_{k\rightarrow\infty}\sup_{t\in [s,r]}W_1(\nu_{i_k}(t),\nu^*(t))=\lim_{k\rightarrow\infty}\sup_{t\in [s,r]}W_1(m_{i_k}(t),m^*(t))= 0.$$ 
	This implies that \begin{equation}\label{A:ineq:nu_star_tau_limit}
	\begin{split}
	W_1(&\nu^*(\tau),\nu^\natural)=W_1(\hat{e}_\tau{}_\#\chi^*,\nu^\natural)\\
	&\leq \lim_{k\rightarrow \infty} [W_1(\hat{e}_{\tau}{}_\#\chi^*,\hat{e}_{\tau_{i_k}}{}_\#\chi^*)
	+W_1(\hat{e}_{\tau_{i_k}}{}_\#\chi^*,\hat{e}_{\tau_{i_k}}{}_\#\chi_{i_k})+ W_1(\hat{e}_{\tau_{i_k}}{}_\#\chi_{i_k},\nu^\natural)]. \end{split}
	\end{equation} Using the fact that the probability $\chi^*$ is concentrated on $2R$-Lipschitz motions, we get that the firs term in the right-hand side of (\ref{A:ineq:nu_star_tau_limit})  is bounded by $2R|\tau_{i_k}-\tau|$. Further, the second term is bounded by $W_1(\chi^*,\chi_{i_k})$. Finally, the third term is equal to $W_1(\nu_{i_k}^\natural,\nu^\natural)$. Taking into account that $\tau_i\rightarrow\tau$, $W_1(\chi^*,\chi_{i_k})$, $W_1(\nu_i^\natural,\nu^\natural)\rightarrow 0$, we conclude from (\ref{A:ineq:nu_star_tau_limit}) that $\nu^*(\tau)=\nu^\natural$.  
	
	Now, let us prove that $\nu^*(\cdot)$ solves (\ref{system:mfdi}) for $m(\cdot)=m^*(\cdot)$. It suffices to prove that $\mathrm{supp}(\chi^*)\subset\mathrm{SOL}(r,s,m^*(\cdot))$. Let $w(\cdot)\in\mathrm{supp}(\chi^*)$. By~\cite[Proposition 5.1.8]{Ambrosio} there exists a sequence $\{w_{i_k}(\cdot)\}\in\mathcal{C}_{s,r}$ such that $w_{i_k}\in\mathrm{supp}(\chi_{i_k})$ and $\|w(\cdot)-w_{i_k}(\cdot)\|\rightarrow 0$ as $k\rightarrow \infty$. Furthermore, since $\mathrm{supp}(\chi_{i_k})\subset\mathrm{SOL}(r,s,m_{i_k}(\cdot))$, we have that, for every $t',t''\in [s,r]$, $t'<t''$,
	$$\mathrm{dist}\left(w_{i_k}(t'')-w_{i_k}(t'),\int_{t'}^{t''}F(t,\mathrm{p}(w_{i_k}(t)),m_{i_k}(t))dt\right) =0.$$ By condition (M\ref{cond:lipschitz}), we have that the functions 
	\begin{equation*}
	\begin{split}
	\Bigl|\mathrm{dist}\Bigl(w(t'')-&w(t'),\int_{t'}^{t''}F(t,\mathrm{p}(w(t)),m^*(t))dt\Bigr)\\
	&-\mathrm{dist}\Bigl(w_{i_k}(t'')-w_{i_k}(t'), \int_{t'}^{t''}F(t,\mathrm{p}(w_{i_k}(t)),m_{i_k}(t))dt\Bigr)\Bigr|\\\leq
	(2+L)\|&w_{i_k}(\cdot)-w(\cdot)\|+L\sup_{t\in [s,r]}W_1(m_{i_k}(t),m^*(t)).
	\end{split} 
	\end{equation*}	Hence, we conclude that for any $t',t''\in [s,r]$, $t'<t''$,
	$$\mathrm{dist}\left(w(t'')-w(t'),\int_{t'}^{t''}F(t,\mathrm{p}(w(t)),m^*(t))dt\right) =0.$$
	This means that $w(\cdot)$ is a solution of (\ref{incl:F}) for $m(t)=m^*(t)$. Therefore, $\nu^*$ solves MFDI (\ref{system:mfdi}) for $m(\cdot)=m^*(\cdot)$.
\end{proof}

\section*{Appendix B. List of notation}
\begin{longtable}{l p{12.0cm}}
	$\mathrm{dist}(x,\Upsilon)$ & distance between the point $x$ and the set $\Upsilon$; see \S \ref{subsect:notation}\\
	$\mathcal{P}(X)$ & the set of probabilities on the metric space $X$; see \S \ref{subsect:notation}\\
	$\mathcal{P}^1(X)$ & the set of probabilities on the metric space $X$  with the finite first moment; see \S \ref{subsect:notation}\\
	$W_1(m_1,m_2)$ & Kantorovich-Rubinstein (1-Wasserstein) distance between probabilities $m_1,m_2\in\mathcal{P}^1(X)$; see (\ref{intro:wass_distance}) \\
	$\mathrm{Lip}_K(X)$ & the set of $K$-Lipschitz continuous functions from the metric space $X$ to $\mathbb{R}$; see \S \ref{subsect:notation} \\
	$h{}_\#m$ & pushforward measure of $m$ by the mapping $h$; see  (\ref{intro:push_forward})\\
	$\pi(\cdot|x)$ & disintegration of the probability $\pi$ defined on $X\times Y$; see (\ref{intro:disintegration})\\
	$\pi_{1,2}*\pi_{2,3}$ & composition of probabilities $\pi_{1,2}$ and $\pi_{2,3}$; see (\ref{intro:composition})\\
	$\td$ & $d$-dimensional torus; see \S \ref{subsect:state}\\
	$\mathrm{p}$ & natural projection of extended state space $\tdr$ on $\td$; see \S \ref{subsect:state}\\
	$\mathcal{C}_{s,r}$ & $C([s,r];\tdr)$; see \S \ref{subsect:state} \\
	$e_t$, $\hat{e}_t$ & evaluation operators from $\mathcal{C}_{s,r}$ to $\td$ and $\tdr$ respectively; see (\ref{intro:evaluation_operators})\\
	$\mathcal{M}_{s,r}$ & the set of measurable functions from $[s,r]$ to $\mathcal{P}^1(\td)$; see \S \ref{subsect:state} \\
	$\mathcal{N}_{s,r}$ & the set of measurable functions from $[s,r]$ to $\mathcal{P}^1(\tdr)$; see \S \ref{subsect:state} \\
	$\widehat{m}$ & lifting of $m\in\ptd$ to $\mathcal{P}^1(\tdr)$; see (\ref{intro:lifting_m})\\
	$[\phi,\nu]$ & the averaging of the function $(x,z)\mapsto \phi(x)+z$ according to $\nu$; see (\ref{intro:action})\\
	$\odot$ & operation of concatenation; see (\ref{intro:concatination_mothions}) and (\ref{intro:concatination})\\
	$F(t,x,m)$ & vectogram of the dynamics in the extended state space; see (\ref{intro:F})\\
	$\widehat{F}(t,w,\nu)$ & $F(t,\mathrm{p}(w),\mathrm{p}_\#\nu)$; see (\ref{intro:hat_F})\\
	$\mathcal{U}_{s,r}$ & set of relaxed controls on $[s,r]$; see \S\ref{subsect:dynamics}\\
	$\mathrm{Sol}(r,s,y,m(\cdot))$ & set of solutions of differential inclusion $(\dot{x}(t),\dot{z}(t))\in F(t,x(t),m(t))$ on $[s,r]$ satisfying  $x(s)=y$; see \S \ref{subsect:dynamics} \\
	$\mathrm{SOL}(r,s,m(\cdot))$ & set of all solutions of differential inclusion $(\dot{x}(t),\dot{z}(t))\in F(t,x(t),m(t))$ on $[s,r]$; see~(\ref{introl:SOL}) \\
	$B^{s,r}_{m(\cdot)}$ & Bellman operator on $[s,r]$ for the control problem with dynamics (\ref{sys:dynamics}), integral part of the  reward function $g$ and the given flow of probabilities $m(\cdot)$; see (\ref{intro:Bellman})\\
	$\Psi^{r,s}$ & set-valued mapping from $\mathcal{P}^1(\td)\times C(\td)$ to itself determining the mean field game dynamics; see Definition \ref{def:flow} \\
	$\mathbb{B}_c$  & closed ball in $\rd$ of the radius $c$ centered at the origin; see \S \ref{sect:viability_theorem}\\
	$\mathcal{L}^c(m)$ & set of probabilities on $\td\times\rdp$ concentrated on $\td\times\mathbb{B}_c\times [-c,c]$ with the marginal distribution on $\td$ equal to $m$; see (\ref{intro:L_c})\\
	$\mathcal{L}^c_*(\nu)$ & set of probabilities on $\tdr\times\rdp$ concentrated on $\tdr\times\mathbb{B}_c\times [-c,c]$ with the marginal distribution on $\td$ equal to $\nu$; see (\ref{intro:L_star_c})\\
	$A^{s,r}_{m}$ & Bellman operator on $[s,r]$ for the control problem with ``frozen'' dynamics; see (\ref{intro:A_Bellman})\\
	$\Theta^\tau$ & shift operator defined on $\td\times\mathbb{R}^{d+1}$ with values in $\tdr$; see (\ref{intro:Theta})\\
	$\mathcal{D}^c_F\mathcal{V}(t,m,\phi)$ & set-valued derivative of the multifunction $\mathcal{V}$ at $t,m,\phi$ by virtue of $F$ under constraints determined by constant~$c$; see Definition \ref{def:tangent}\\
	$\mathrm{BL}_{M,C}$ & set of functions $\phi\in\mathrm{Lip}_C(\td)$ bounded by $M$; see \S \ref{sect:viability_theorem}\\
	$\Xi^\tau$ & shift operator defined on $\tdr\times\mathbb{R}^{d+1}$ with values in $\tdr$; see (\ref{intro:Xi}) \\ 
	$\Delta^{s,r}$ & difference operator; see (\ref{intro:Delta}) 
	
\end{longtable}

\begin{acknowledgement}
	I would like to thank the referee for his/her valuable comments.
\end{acknowledgement}

\bibliographystyle{abbrv}
\bibliography{viab_mfg} 

\end{document}